\documentclass[12pt,psfig]{article}
\usepackage{amsthm,amssymb,amsmath,amsfonts,amscd}
\usepackage{graphicx,epsfig,latexsym,hyperref,epstopdf}
\usepackage{mathtools}
\usepackage[space]{grffile}
\usepackage{subfig,pgfplots}
\usepackage{circuitikz}
\usepackage{listings}
\usepackage{caption}
\usepackage{hyperref}
\usepackage{cleveref}
\usepackage{enumitem}
\usepackage{cite,calc,xcolor,mathrsfs,dsfont}
\usepackage[nottoc,numbib]{tocbibind}
\usepackage{xcolor}
\usepackage{harmony}
\usepackage{float}
\usepackage{graphicx}

\newcommand{\cb}{\pmb{c}}

\newcommand{\sinc}{{\rm sinc}}

\newcommand{\diag}{{\rm diag}}
\newcommand{\0}{\mathbf{0}}
\newcommand{\x}{\mathbf{x}}

\newcommand{\y}{\mathbf{y}}

\newcommand{\ie}{{\em i.e.,} }
\newcommand{\eg}{{\em e.g.,} }
\newcommand{\hot}{{\rm h. o. t. }}

\newcommand{\mylvert}{\left\vert\!\left\vert}
\newcommand{\myrvert}{\right\vert\!\right\vert}
\newtheorem{thm}{Theorem}[section]
\newtheorem{cor}[thm]{Corollary}
\newtheorem{lem}[thm]{Lemma}

\theoremstyle{definition}
\newtheorem{defn}[thm]{Definition}
\newtheorem{exm}[thm]{Example}
\newtheorem{rem}[thm]{Remark}
\newtheorem{notation}[thm]{Notation}
\numberwithin{equation}{section}
\def\Blem {\begin{lem}}
\def\Elem {\end{lem}}
\def\be {\begin{equation}}
\def\ee {\end{equation}}
\def\ba {\begin{eqnarray}}
\def\ea {\end{eqnarray}}
\def\bes {\begin{equation*}}
\def\ees {\end{equation*}}
\def\bas {\begin{eqnarray*}}
\def\eas {\end{eqnarray*}}
\def\bpr {\begin{proof}}
\def\epr {\end{proof}}
\begin{document}
\baselineskip=18pt
\renewcommand {\thefootnote}{ }

\pagestyle{empty}

\begin{center}
\leftline{}
\vspace{-0.500 in}
{\Large \bf 
Musical tone coloring via bifurcation control of Eulerian \(n\)-tuple Hopf singularities
} \\ [0.3in]

{\large Majid Gazor\(^{\dag}\)\footnote{$^\dag\,$Corresponding author. Phone: (98-31) 33913634; Fax: (98-31) 33912602;
Email: mgazor@iut.ac.ir; Email: ahmad.shoghi@math.iut.ac.ir.} and Ahmad Shoghi }

\vspace{0.105in} {\small {\em Department of Mathematical Sciences, Isfahan University of Technology
\\[-0.5ex]
Isfahan 84156-83111, Iran }}

\today

\vspace{0.05in}

\noindent
\end{center}

\vspace{-0.10in}

\baselineskip=16pt

\:\:\:\:\ \ \rule{5.88in}{0.012in}

\begin{abstract}
An intrinsic essence of sounds in music is the evolution of their qualitative types while in mathematics we interpret each {\it qualitative change} by a {\it bifurcation}. {\it Hopf bifurcation} is an important venue to generate a signal with an arbitrary frequency. Hence, the investigations of musical sounds via bifurcation control theory are long-overdue and natural contributions. In this paper, we address the tone coloring of sounds by dynamical modeling of spectral and temporal envelopes. Multiple number of leading harmonic partials of a note (modulo a hearing sound velocity threshold) are attributed into an Eulerian differential system with \(n\)-tuple Hopf singularity. The qualitative evolution of the temporal envelop is then simulated over a set of consecutive time-intervals via bifurcation control of the differential system. For an instance, our proposed approach is applied on audio \(C\sharp 4\) files obtained from piano and violin. Fourier analysis is used to generate the amplitude spectral vectors. Then, we associate each amplitude spectral vector with an Eulerian flow-invariant leaf. Bifurcation control suffices to accurately construct the desired spectral and amplitude envelopes of musical notes. These correspond with a rich bifurcation scenarios involving Clifford toral manifolds for the Eulerian differential system. In order to reduce the technicalities, we employ several reduction techniques and use one bifurcation parameter. We show how different ordered sets of elementary bifurcations such as pitchfork and (double) saddle-node bifurcations are associated with the qualitative temporal envelop changes of a \(C\sharp 4\) played by either a piano or a violin. A complete hysteresis type cycle is observed within the temporal envelop bifurcations of the \(C\sharp 4\) played by violin.
\end{abstract}



\vspace{0.10in} \noindent {\it 2010 Mathematics Subject Classification}:\, 34H20; 00A65; 55U10; 57N15; 34C23.




\section{Introduction }
Qualitative changes in sounds are the intrinsic elements of music. A qualitative change in mathematics is called a {\it bifurcation}. A bifurcation occurs when parameters of a system vary around their critical values and the qualitative behavior of the system changes. Changes in the number and stabilities of equilibria, limit cycles, (minimal) flow-invariant manifolds are each called a {\it bifurcation}. The musical instances of bifurcations occur via changes in timber (tone color), tempo, sound frequencies, dynamics and the number of notes in a chord. Hence, the dynamical modeling and simulation using bifurcation control provides a natural method for the mathematical understanding of these qualitative type patterns in musical art. In the existing literature, there are some worthwhile contributions on music and model based system theory; \eg see \cite{DynamicSpectralEnvelopeModeling,DynamicsPiano,Musicology1997,MusBoon1995,MusLeman1990,MusGeorgescu1990,MusLarge2010,MusSturman2012,KossMusicDance}. There are also an extensive recent literature on bifurcation control theory; \eg see \cite{GazorSadriBT,GazorShoghiEulNF,GazorSadri,ChenBifControl2000,ChenBifuControl,Colonius3,Colonius2,Colonius1,BifStablizationSICON99,KrenerControlBif,HamziKangCenter05,HamziCharNormalF,Respondek2006,
KangKrener,KangIEEE,Kang200,Kang98,Kang04}. However, there is no any contribution dealing with the study of musical sounds and bifurcation control. Our goal here is to present a differential system as a natural mathematical model to construct sounds of notes in musical art. We address the tone coloring of musical sounds using spectral and temporal envelopes. These are two of the most important factors in tone coloring of sounds.

There are at least two advantages for our investigation in this paper. The first is to get a mathematical perception about the {\it timbral qualitative changes} in music using bifurcation control theory. The second is to associate an audio system to the bifurcations of a family of differential systems for their better understanding by {\it ``\textbf{hearing} how bifurcations \textbf{sound} in music}''. Hence, {\it theorems on these bifurcations} can be {\it best played by musicians} and then, one can hear the {\it sound of the theorems}; also see \cite{TheoremSound}. For instance, in Section \ref{SecleafBifurcations} we show that bifurcations of asymptotically stable flow-invariant \(n\)-tori \(\mathbb{T}_n\) for large values of \(n\) naturally correspond with temporal and spectral envelops of notes in music. These are invariant geometrical objects that can be heard when the differential systems are {\it simulated and sounded by a {\sc Matlab} programming on a computer}.

\pagestyle{myheadings} \markright{{\footnotesize {\it M. Gazor and A. Shoghi \hspace{2.8in} {\it Musical timber and bifurcation control}}}}

The human ear is known as a \emph{biological frequency spectrum analyzer}. Cochlea are naturally partitioned into regions specific to different frequencies. For instance, the apex of cochlea {\it processes low frequencies} while its base deals with {\it high frequencies}. This is a natural decomposition of a harmonic sound into its harmonic partials by the ear and contributes to the {\it superposition principle} of the sound. Human ear naturally combines a set of synchronised sounds with different frequencies into a complex tone \cite{TuningTimbreSpectrumScale}. 

Harmonic partials carry the {\it partial information} from the full data attributed to a note. A musical tone (modulo {\it a hearing sound velocity threshold}) is a {\it superposition} of multiple partials with different frequency, amplitude, and phase; \eg see \cite{FundamentalsOfMusicProcessing2015}. These partials can be simply combined by sounding the partials at the same time. This is technically known as {\it additive synthesis}; see \cite{TheMathematicalFoundationsOfMusic}. Thus, a single note sounds the same as simultaneously playing its dominant partials; see \cite{GoodVibrations2009}. The {\it superposition principle} of sounds is employed in modeling a musical tone as of a solution trajectory for a differential system. Since we intend to generate partials using Hopf bifurcations, the state variables have to come in pairs (say, (\(x_i, y_i\)) for \(i\leq n\)) and the state dimension must be an even number. Each pair of the state variables generates an individual partial of the musical note. Due to the number of required partials for a precise estimation of tone color, the state dimension can be very large and the approach must be organised to deal with such differential systems with large dimensions. Hopf bifurcations with large number of multiplicities undergo a high codimension singularity and require multiple controllers to fully control their dynamics. This is not a desirable approach for our ends. Given the structural symmetry of Eulerian flows, the partials of each note are naturally well-defined and synchronized in our modeling approach. Yet, the full bifurcation analysis of Eulerian flows is involved with CW complex structures from algebraic topology and it appeals to technical tools such as normal forms, cell-decompositions and flow-invariant foliations. These are skipped here for simplicity and the interested reader are referred to \cite{GazorShoghiBifuControl}.


To simplify the exposition, we here apply our proposed approach on a particular Eulerian system with one bifurcation parameter and illustrate its dynamics in simple terms by appealing to minimal technical concepts. This particular Eulerian system suffices to fully address the temporal amplitude control bifurcations and the spectral envelop simulation of musical notes in this paper. The bifurcations are presented for the flows restricted on a flow-invariant manifold (called by leaf-manifold). We comprehensively address the magnitude spectrum tuning and associate the invariant leaf-manifolds with the magnitude spectra of musical sounds. Next, we consider the first pair of the state variables corresponding with the leading harmonic partial on the leaf manifold. Then, we decouple it from the rest of state variables and ignore its angular component to obtain a scalar differential equation. There is a one-to-one correspondence between nonnegative steady-state bifurcations from the scalar differential equation with the ongoing flow-invariant hypertorus bifurcations on the leaf-manifold. This simply correlates the temporal envelop bifurcations with elementary bifurcations from one-dimensional state-space differential equations. We show how an ordered set of bifurcations in one-dimensional state space are associated with the qualitative type changes in the temporal envelops of notes. For an instance, a consecutive round of bifurcations consisting pitchforks and double saddle-nodes contribute into a hysteresis type phenomena. These steady-state bifurcations correspond with minimal invariant hypertori for the Eulerian flows. This is how they describe {\it temporal envelop bifurcations} in timbral changes of a \(C\sharp 4\) pitch played by piano. The discussion on the spectral bifurcations caused by qualitative changes in magnitude spectrum is an in-progress project (\eg see \cite{GazorShoghiHarmonic}) and thus, they are skipped in this paper. The bifurcations corresponding with the magnitude spectrum are generally far more complex than the temporal envelop bifurcations and are involved with bifurcations of flow-invariant toral manifolds (called by toral CW complexes); see \cite{GazorShoghiBifuControl}. This paper is the third part of our project on the bifurcation control of Eulerian flows with multiple Hopf singularities and applications in robotic team control, tone coloring of musical notes and harmonic music design; see \cite{GazorShoghiRobotic,GazorShoghiHarmonic,GazorShoghiEulNF,GazorShoghiBifuControl}.


This paper is organized as follows. Section \ref{Sec2} discusses temporal and spectral envelops. A one-parametric polynomial family of Eulerian differential systems are considered in Section \ref{Sec3}. We apply several reduction techniques and make a comprehensive bifurcation analysis. In Section \ref{Sec4}, the criteria for accurate peak estimations of the amplitude spectrum are presented. An Eulerian bifurcation control of musical timber is introduced in Section \ref{SecModel}. Section \ref{Algorthm} provides an algorithm on how to simulate the temporal envelop bifurcations. The algorithm is illustrated in Section \ref{SecPianoCons} using audio files of \(C\sharp4\) played by piano and violin. Finally, Section \ref{SecleafBifurcations} summarizes how different bifurcations of flow-invariant hypertori occurs when \(C\sharp 4\) is played by either piano or violin.

\section{Temporal envelop bifurcation and spectral envelop simulation }\label{Sec2}

There are several tone color modeling approaches such as {\it multidimensional scaling}, {\it analogies with vowels}, a {\it pragmatic synthesis
approach}, and {\it sound envelops}, etc; \eg see \cite{TuningTimbreSpectrumScale}. The \emph{temporal envelop bifurcation} and \emph{spectral envelope} simulations contribute into tone coloring (timber) of a musical sound; \eg see \cite{DynamicSpectralEnvelopeModeling}. This paper is devoted to address these two key timbral aspects for computer generated musical sounds. These are simulated by Eulerian bifurcations with multiple Hopf singularities using computer programming in {\sc Matlab}.

\begin{defn}[Temporal envelop bifurcations]
A temporal (amplitude) envelope of a note consists of dynamics changes of the sound over time and may be partitioned by four time segments: \emph{attack}, \emph{decay}, \emph{sustain} and \emph{release}; \eg see Figure \ref{CsharpP_Sim} and \ref{CsharpV_Sim}.
The attack time is generally activated by pressing the musical instruments until (before) the amplitude starts to decrease. The decay time finishes when the dynamics retains a sustainable level. The sustain amplitude level usually is viable for a longer period of time than the attack and decay times. The sustain time is terminated once the musical key is released; \eg see \cite{TheMathematicalFoundationsOfMusic,TuningTimbreSpectrumScale}. These gives rise to {\it qualitative changes} of sound and thus, they are called by {\it temporal envelop bifurcations}. These temporal envelop bifurcations play a central role in {\it tone coloring} and {\it spatial hearing}. Temporal amplitude bifurcations are associated with various type of minimal flow-invariant hypertorus bifurcations of the corresponding differential system.
\end{defn}

\begin{rem}\label{Rem4.12}(Hearing threshold)
Human audible frequencies are from \(20\)Hz to \(20k\)Hz. The upper frequency threshold decreases when people get older. {\it Threshold of hearing} (\textrm{dB}) usually refers to the lowest audible intensity corresponding with pure tones in a noiseless environment; see \cite{DigitalSignalProcessing,ThePhysicsOfMusicAndColor}.
Audio files in {\sc Matlab} are read by the command \texttt{[y, Fs] = {\tt audioread(filename, dataType)}}. When data type of an audio file is taken as `double',
the elements of the matrix \texttt{y} are normalized between $-1.0$ and $1.0.$ The `native' data type refers to different varieties of scaling data in {\sc Matlab}. In either cases, the amplitude of recorded sound wave does not provide any information about sound pressure level of the audio signal. Thereby, we appeal to the {\it (MIDI) velocity values} with an integer range of data between 0 and 127. Here, the zero velocity is equivalent to a note-off. We assume that the dynamic {\it\textbf{ppp}} (pianississimo) is the softest dynamic in our tone modeling with an approximate velocity of 16. Therefore, we propose the {\it threshold of hearing} as less than half of this. Hence, the normalized velocity value of \(\frac{7}{127}\thickapprox 0.0551\) is taken as the {\it threshold of hearing} throughout this paper. Thus, estimations with errors of less than this are considered as accurate.
\end{rem} Every note on a stave is a musical sound specified by four characteristic: {\it pitch, duration, dynamics} and {\it tone color (timber)}. The sinusoids in the Fourier series analysis of a sound is called by a {\it harmonic partial} or simply a {\it partial}. The leading partial {\it frequency of a note} corresponds with the {\it pitch} of the note. This is the lowest partial frequency in a tone and is called by {\it fundamental frequency}. The frequency of the remaining partials are integer multiples of the fundamental frequency and they are called by {\it overtones} with frequency \(\omega_k\), where \(\omega_k=k\omega_1\) for \(2\leq k\leq n.\) The {\it duration} of a note refers to the time interval that the note is played in a sheet music. {\it Dynamics} refers to the loudness of the note. The {\it sound quality} is called by {\it timber} or {\it color of a sound}. Timber depends on the method, materials and shape of the vibrating instrument; see \cite{TheoryPractice,AcousticsAndPsychoacoustics}. The {\it harmonic (partials) series} of a sound determine the sound and the spectral envelop of its timber; see Figure \ref{D4-Csharp4Spectrum}. Remark that bells and tympani have inharmonic partials. However, inharmonic partials, noises and vibrations are not treated in this paper.


\begin{defn}[Spectral envelop and the peaks of the amplitude spectrum]
Let \(\mathcal{F}(f)\) be the Fourier transform of a musical note signal \(f(t).\) Then, the amplitude spectrum (Fourier magnitude spectrum) is \(|\mathcal{F}(f)(\omega)|.\) Let \(\omega_i\) be the \(i\)-th partial frequency of the signal \(f(t).\) A spectral envelope of \(f(t)\) is a frequency domain function and is defined as a function interpolating the points
\be\label{Peaks} \left\{\big(\omega_i, |\mathcal{F}(f)(\omega_i)|\big)\, |\, \omega_i \hbox{ is a harmonic partial frequency of } f(t) \hbox{ for } 1\leq i\leq n \right\}.\ee
The points in the set \eqref{Peaks} are called by the \(n\)-{\it leading peaks} (we simply refer to them by the {\it peaks}) of the amplitude spectrum. The vector \((|\mathcal{F}(f)(\omega_i)|)^n_{i=1}\) plays the role of the amplitude spectral vector \({\pmb{c}}\) and corresponds with a flow-invariant leaf-manifold.
\end{defn} We here consider the maximum number of required partials for modeling musical tones as \(n:=6.\) This is consistent with the claim on
\cite[page 126]{PhysicsAndMusic}, where authors indicate that the first five harmonics are sufficient for modeling a complex tone with a fundamental frequency less than \(200\)Hz. The number of dominant harmonic partials decreases as the fundamental frequency increases.

\section{Eulerian differential systems: flow-invariant leaves and bifurcations}\label{Sec3}
\begin{notation}
Let \(\mathbf{a}:=(a_i)^n_{i=0}\) denote the \(n+1\)-dimension real vector \((a_0, a_1, \ldots, a_n)\in \mathbb{R}^{n+1}.\) Further, write \(\cos \mathbf{a}\) instead of the real vector \((\cos a_0, \ldots, \cos a_{n})\) and denote \(\sin\mathbf{a}\) for \((\sin a_i)^n_{i=0}.\) Denote \(\mathbb{R}[
\mu, \x, \y]\) for the set (ring) of all polynomials in terms of \(\x, \y\in \mathbb{R}^{n+1}\) and \(\mu\in \mathbb{R}.\) A hypertorus refers to a manifold homeomorphic to the standard \(n\)-torus and is denoted by \(\mathbb{T}_{n}\).
\end{notation}
Every polynomial Eulerian vector field is denoted by
\bes
E_f:=f(\mu, \x, \y)E_{\0}\quad \hbox{ for some } \; f\in \mathbb{R}[\mu, \x, \y], \; \hbox{ where } \; E_{\0}:=((x_i)^n_{i=0}, (y_i)^n_{i=0}).
\ees 
Define the linear \emph{rotating vector field} \(\Theta^{\pmb{\omega}}_{\0}\) with angular frequency vector \(\pmb{\omega}\) by
\bes \Theta^{\pmb{\omega}}_{\0}:=\left((-\omega_i y_i)^n_{i=0}, (\omega_i x_i)^n_{i=0}\right),\quad \hbox{ where }\; \pmb{\omega}:= (\omega_0, \omega_1, \ldots, \omega_n)\in \mathbb{R}^{n+1}.\ees
An initial value problem associated with a polynomial Eulerian differential system and a $n$-tuple Hopf singularity refers to
\begin{eqnarray}\label{EulerianDiff}
&\frac{\textrm{d}}{\textrm{d}t}(\x, \y):= \Theta^{\pmb{\omega}}_{\0}+E_f= \left((-\omega_i y_i)^n_{i=0},(\omega_i x_i)^n_{i=0}\right)+f(\mu, \x, \y) E_{\0}, \; \x(0, \mu)=\x^{\circ}, \y(0, \mu)=\y^{\circ}, &
\end{eqnarray} for \(f\in \mathbb{R}[\mu, \x, \y].\) The pairs \((\omega_i, |(\x^\circ_i, \y^\circ_i)|_2)\) for \(1\leq i\leq n\) correspond with the peaks of the spectral envelop of a musical sound while the function \(f\) is modeled based on its temporal envelop. The Eulerian vector field \(E_0\) generates synchronised harmonic partials. Thus, the individual state variables naturally can simulate the harmonic partials of a given musical sound. This usage is also enabled by the superposition principle property of sounds. The oscillation of partials are generated via \(\Theta^{\pmb{\omega}}_{\0}.\) The initial value problem \eqref{EulerianDiff} in polar coordinates is read by
\begin{eqnarray}\label{PolarEq1}
&\frac{\mathrm{d} \mathbf{r}}{\mathrm{d}t} =\mathbf{r} f\left(\mu, r_0\cos\theta_0, \ldots, r_n\cos\theta_n, r_0\sin\theta_0, \ldots, r_n\sin\theta_n \right),  \quad \dot{\theta}_i=\omega_i, &
\end{eqnarray} where \(r_i(0, \mu)=({x_i^{\circ}}^2+{y_i^{\circ}}^2)^{\frac{1}{2}}\) and \(\theta_i(0)= \tan^{-1}\frac{y_i^\circ}{x_i^\circ}\) for \(0\leq i\leq n.\) This is deduced from \(r_i\dot{r}_i=x_i\dot{x}_i+y_i\dot{y}_i,\) \({r_i}^2\dot{\theta}_i=x_i\dot{y}_i-y_i\dot{x}_i\) and
\begin{eqnarray}\label{x_i, y_i} &x_i(t, \mu)=r_i(t, \mu)\cos\left(\omega_it+\arctan\frac{y_i^{\circ}}{x_i^{\circ}}\right), \quad
  y_i(t, \mu)=r_i(t, \mu)\sin\left(\omega_it+\arctan\frac{y_i^{\circ}}{x_i^{\circ}}\right).&
\end{eqnarray}

Denote \((\mathbf{r}(t, \mu), \theta(t))\) for a solution trajectory of the differential system \eqref{PolarEq1}. For every \({\pmb{c}}\in \mathbb{R}^{n+1}_{\geq0},\)  there exists a maximal half-open interval \(E^{loc}_{\mu,{\pmb{c}}}=[0, e_{\mu, \cb})\) for \(e_{\mu, \cb}\in \mathbb{R}\cup\{\infty\}\) such that \(\mathbf{r}(0, \mu)=\cb,\) and \(\mathbf{r}(t, \mu)\) exists for all \(t\in E^{loc}_{\mu,{\pmb{c}}}\). We call \(E^{loc}_{\mu,{\pmb{c}}}\) by the maximal forward-time interval of the existence for the trajectory \((\mathbf{r}(t, \mu), \theta(t))\). Now define the local manifold \(\mathcal{M}^{loc}_{\mu, {\pmb{c}}}\)  by
\begin{eqnarray}\label{ManifLocal}
&\mathcal{M}^{loc}_{\mu,{\pmb{c}}}:= \{(\mathbf{u}, \mathbf{v})\in \mathbb{R}^{2n+2}|\,  \|(u_i, v_i)\|_2=c_i\xi(\mu,t)\, \hbox{ for }\, 0\leq i\leq n \; \hbox{ and }\, t\in E^{loc}_{\mu,{\pmb{c}}}\}, &
\end{eqnarray} where
\be\label{xi}\xi(\mu, t):=\exp\left(\int_{0}^{t}f\left(\mu, \x\big(\mathbf{r}(\tau, \mu), \theta(\tau)\big), \y\big(\mathbf{r}(\tau, \mu), \theta(\tau)\big)\right)d \tau\right).
\ee In section \ref{SecModel}, the vector \({\pmb{c}}\) is associated with the amplitude spectrum evaluated at partial frequencies of a note. Thus, we call \(\cb\) by the amplitude spectral vector. Further, the individual segments of the timbral dynamics of a note lies on the space \(\mathcal{M}^{loc}_{\mu,{\pmb{c}}}\) while the whole timbral dynamics of a note lives on the leaf manifold
\be\label{Mc}\mathcal{M}_{{\pmb{c}}}:= \cup_{\alpha \in \mathbb{R}_{>0}}\mathcal{M}^{loc}_{\mu, \mathbf{\alpha c}}.\ee The leaf manifold \(\mathcal{M}_{{\pmb{c}}}\) is independent of the parameter \(\mu\) (see claim \ref{welldef} in Lemma \ref{lem1}) and is homeomorphic to \(\mathbb{R}\times \mathbb{T}_{n+1}\); see \cite{GazorShoghiBifuControl}.

\begin{lem}\label{lem1} Let \((\x(t, \mu), \y(t, \mu))\) be a solution trajectory for the initial value problem \eqref{EulerianDiff}. Then,

\begin{enumerate}
\item \label{Item1-1}There exists \({\pmb{c}}\in \mathbb{R}^{n+1}_{\geq0}\) such that the manifold \(\mathcal{M}^{loc}_{\mu, {\pmb{c}}}\) given by \eqref{ManifLocal} is flow-invariant under the trajectories of the differential system \eqref{EulerianDiff}.
\item\label{welldef} The set \(\cup_{\alpha \in \mathbb{R}_{>0}}\mathcal{M}^{loc}_{\mu, \mathbf{\alpha c}}\) is independent of the parameter \(\mu\) and the function \(f\). In particular, \(\cup_{\alpha \in \mathbb{R}_{>0}}\mathcal{M}^{loc}_{\mu_1, \mathbf{\alpha c}}= \cup_{\alpha \in \mathbb{R}_{>0}}\mathcal{M}^{loc}_{\mu_1, \mathbf{\alpha c}}\) for all \(\mu_1, \mu_2 \in \mathbb{R}.\) Furthermore, \(\mathcal{M}_{{\pmb{c}}}= \mathcal{M}_{\alpha{\pmb{c}}}\) for all \(\alpha\in \mathbb{R}_{>0}.\)
\item\label{Item3} The leaf manifold \(\mathcal{M}_{{\pmb{c}}}\) is invariant under \((\x(t, \mu), \y(t, \mu))\) if and only if \(\mathbf{r}(t, \mu)=\frac{{r}_{1}(t, \mu)}{c_1}{\pmb{c}}.\)
\end{enumerate}
\end{lem}
\begin{proof} Let \(\cb\neq 0\). Claim \eqref{Item1-1}: From equation \eqref{PolarEq1}, for every \(0\leq i\leq n\) we have
\begin{eqnarray*}
\ln(|r_i(t, \mu)|)-\ln(|r_i(0, \mu)|)=\int_{0}^{t}f(\mu, \x(\mathbf{r}(\tau, \mu), \theta(\tau)), \y(\mathbf{r}(\tau, \mu), \theta(\tau)))\mathrm{d}\tau.
\end{eqnarray*} Substituting \(|r_i(t, \mu)|=\|(x_i(t, \mu), y_i(t, \mu))\|_2\) into the above relation gives rise to
\begin{eqnarray}\label{Eq2}
\|(x_i(t, \mu), y_i(t, \mu))\|_2=\|(x_i(0, \mu), y_i(0, \mu))\|_2\exp\left(\int_{0}^{t}f(\mu, \x(\mathbf{r}, \theta), \y(\mathbf{r}, \theta))\mathrm{d}\tau\right).
\end{eqnarray} Now it suffices to take \(c_i:=\|(x_i^{\circ}, y_i^{\circ})\|_2\) and \(\xi(t, \mu)\) as given by equation \eqref{xi}.

The claim \ref{welldef} directly follows from
\begin{eqnarray*}
\mathcal{M}_{{\pmb{c}}}=\{(\mathbf{u}, \mathbf{v})\in \mathbb{R}^{2n+2}\, |\, c_j\|(u_i, v_i)\|_2=c_i\|(u_j, v_j)\|_2 \text{\quad for\quad} 0\leq i\leq j\leq n\}.
\end{eqnarray*}

Claim \ref{Item3}: Let \(\mathcal{M}_{{\pmb{c}}}\) be invariant under the trajectory \((\x(t, \mu), \y(t, \mu))\). By equation \eqref{Eq2},
\begin{eqnarray*}
&\mathbf{r}=(r_0, r_1, \ldots, r_n)= \left(\frac{r_0(t, \mu)}{c_1}c_0, {r_1(t, \mu)}, \ldots, \frac{ r_1(t, \mu)}{c_1}c_n\right)=\frac{r_1(t, \mu)}{c_1}(c_0,c_1, \ldots, c_n)=\frac{r_1(t, \mu)}{c_1}{\pmb{c}}.&
\end{eqnarray*}
Now assume that \(\mathbf{r}(t, \mu)=\frac{{r}_{1}}{c_1}{\pmb{c}}.\) By claim \ref{Item1-1}, there exists \({\pmb{c}}^{\prime}\) so that \(\mathcal{M}_{{\pmb{c}}^{\prime}}\) is invariant under the flow \((\x, \y).\) This enforces the equalities
 \begin{eqnarray*}
 &\frac{c_i^{\prime}}{c_1^{\prime}}=\frac{c_i}{c_1}=\frac{r_i(t, \mu)}{r_1(t, \mu)} \; \hbox{ and }\; {\pmb{c}}^{\prime}=(c_0^{\prime}, c_1^{\prime}, \ldots, c_n^{\prime})=(\frac{c_1^{\prime}}{c_1}c_0, \ldots, \frac{c_1^{\prime}}{c_1}c_n)=\frac{c_1^{\prime}}{c_1}{\pmb{c}}.&
 \end{eqnarray*}
Since \(\mathcal{M}_{c_1^{\prime}{\pmb{c}}/c_1}^{loc}\subset\mathcal{M}_{{\pmb{c}}}\), the manifold  \(\mathcal{M}_{{\pmb{c}}}\) is invariant under the trajectory \((\x, \y).\)
\end{proof}

By Lemma \ref{lem1}, the reduced system \eqref{EulerianDiff} on the invariant manifold \(\mathcal{M}_{{\pmb{c}}}\) follows
\ba\label{ReducedEqu1}
&\textstyle \frac{d\mathbf{r}}{dt}= \frac{r_1}{c_1}\cb f\left(\mu, \frac{c_0}{c_1}r_1\cos\theta_0(t), \ldots,\frac{c_n}{c_1}r_1\cos\theta_n(t), \frac{c_1}{c_1}r_1\sin\theta_0(t), \ldots, \frac{c_n}{c_1}r_1\sin\theta_n(t) \right),&\\\label{ReducTheta}
&\theta_i(t)=\omega_it+\arctan\frac{y_i^{\circ}}{x_i^{\circ}}\quad \hbox{ for }\; t\in E^{loc}_{\mu, \cb}= [0, e_{\mu, \cb}), \cb\neq 0, r_i(0, \mu)=({x_i^{\circ}}^2+{y_i^{\circ}}^2)^{\frac{1}{2}}.&
\ea Let \(f\) be a \(\mathbb{T}_{n+1}\)-invariant polynomial. Then,
\(f\in \mathbb{R}[\mu, {x_0}^2+{y_0}^2, \ldots, {x_n}^2+{y_n}^2]\) and there exists a function
\bes
\widehat{f}:\mathbb{R}\times\mathbb{R}^{n+1}\rightarrow \mathbb{R}\; \quad \hbox{ such that }\; \widehat{f}(\mu, {x_0}^2+{y_0}^2, \ldots, {x_n}^2+{y_n}^2)= f(\mu, \x, \y).
\ees In this case, the coupled reduced system \eqref{ReducedEqu1}-\eqref{ReducTheta} on \(\mathcal{M}_{\pmb{c}}\) follows equations \eqref{ReducTheta} and
\ba \label{ReducedEqu2}
& \frac{d\mathbf{r}}{dt}= \frac{r_1}{c_1}\cb \tilde{f}(\mu,r_1), \quad\hbox{ where } \quad \tilde{f}(\mu,r_1):=\widehat{f}\left(\mu, \frac{{c_0}^2}{{c_1}^2}{r_1}^2, {r_1}^2, \ldots,  \frac{{c_n}^2}{{c_1}^2}{r_1}^2 \right).  &
\ea Thus, the angular components are decoupled from the amplitude equations. This gives rise to an effective reduction technique by ignoring the angular components. This is employed in the next lemma. Nonlinear transformations can transform all systems of type \eqref{ReducedEqu1} into a system of the form \eqref{ReducTheta}-\eqref{ReducedEqu2}. The differential system \eqref{ReducTheta}-\eqref{ReducedEqu2} is called a normal form system; \eg see \cite{GazorShoghiEulNF}.

\begin{lem}\label{lem2} Let \(\0\neq {\pmb{c}}\in \mathbb{R}^{n+1}_{\geq0}\). Let the hypotheses of Lemma \ref{lem1} hold, \(f\in \mathbb{R}[\mu,{x_0}^2+{y_0}^2, \ldots, {x_n}^2+{y_n}^2],\) and consider the differential equation
\begin{eqnarray}\label{k-thEqPolarReduced}
& \frac{\mathrm{d} r_1}{\mathrm{d}t}=r_1 \tilde{f}(\mu, r_1).&
\end{eqnarray}
\begin{enumerate}
\item\label{monotonicity} Every solution \(r_i(t, \mu)\) for equation \eqref{ReducedEqu2}, \(0\leq i\leq n,\) is a \(\mu\)-parametric family of monotonic
functions. Furthermore, \(r_i(t, \mu)\) is positive (negative) for all \(t\in E^{loc}_{\mu,{\pmb{c}}}\) if and only if \(r_1(0, \mu):=r_1^{\circ}>0\) \((r_1^{\circ}<0)\).
\item\label{claim1} When the solution \(r_1(t, \mu)\) of \eqref{k-thEqPolarReduced} approaches to an equilibrium in forward time, the maximal forward-time
interval of the existence for \(r_1\) is the positive real line.
\item\label{claim3} An equilibrium of the differential system \eqref{ReducedEqu2} on the manifold \(\mathcal{M}_{\pmb{c}}\) is given by \(\frac{r^*_{1}}{c_1}{\pmb{c}}\)
when it exists. Further, \(\frac{r^*_{1}}{c_1}{\pmb{c}}\) is an equilibrium for \eqref{ReducedEqu2} if and only if \(r^*_{1}\) is an equilibrium for the scalar equation \eqref{k-thEqPolarReduced}. The point \(\frac{r^*_{1}}{c_1}{\pmb{c}}\) is an asymptotically stable (unstable) equilibrium for  \eqref{ReducedEqu2} if and only if \(r^*_{1}\) is an asymptotically stable (unstable) equilibrium for  \eqref{k-thEqPolarReduced}.
\item\label{claim4} Let \(\omega_{n_j}= 0\) and \(\omega_i\neq 0\) for \(i\neq {n_j}\) and \(0\leq j< k\leq n\). Then,
the manifold
\begin{eqnarray}\label{eq4.12}
&\small\Gamma:=\left\{\frac{r_1^{\star}}{c_1}(\cos(\mathbf{a})\,\diag({\pmb{c}}), \sin(\mathbf{a})\,\diag({\pmb{c}}))\in \mathbb{R}^{2n+2}\,|\,  \mathbf{a}\in\mathbb{R}^{n+1}, a_{n_j}=\theta^\circ_{n_j}\,\hbox{ for }\, j\leq k \right\}&
\end{eqnarray} is an invariant (\(n+1-k\))-torus under the flow of the system \eqref{EulerianDiff} if and only if \(r^*_{1}\) is an equilibrium for equation \eqref{k-thEqPolarReduced}. Furthermore, \(\Gamma\) is \(\mathcal{M}_{\pmb{c}}\)-asymptotically stable (unstable) (\ie only for the flows on \(\mathcal{M}_{\pmb{c}}\)) if and only if \(r_1^{\star}\) is an asymptotically stable (unstable) equilibrium for \eqref{k-thEqPolarReduced}.
\end{enumerate}
\end{lem}
\begin{proof}
Claim \ref{monotonicity}. By equation \eqref{Eq2}, we have
\begin{eqnarray*}
&r_i(t, \mu)= \frac{c_i}{c_1}r_1^{\circ}\exp\left(\int_{0}^{t}\tilde{f}(\mu, r_1)\mathrm{d}\tau\right).&
\end{eqnarray*}
Hence, \(r_1^{\circ}r_i(t, \mu)>0\) for \(0\leq i\leq n\) and all \(t>0\). Now suppose that the nonzero solution \(r_1(t, \mu)\) is not monotonic. Therefore, there exist \(t_1\) and \(t_2\) such that \(0<t_1 <t_2\),
\begin{eqnarray}\label{MonotonicProof}
&\frac{\mathrm{d}}{\mathrm{d}t}r_1(t_1, \mu)=0 \text{\quad and \quad} \left\vert r_1(t_1, \mu)-r_1(t_2, \mu) \right\vert>0.&
\end{eqnarray}
Denote \(\varphi(t, \mu, r_1^{\circ})\) for the flow of \eqref{k-thEqPolarReduced}, where \(\varphi(0, \mu, r_1^{\circ})=r_1^{\circ}\). Thus, \(\varphi_1(t_1, \mu, r_1^{\circ})\) is an equilibrium for \eqref{k-thEqPolarReduced} due to \(\tilde{f}(\mu, \varphi_1(t_1, \mu, r_1^{\circ}))=0.\) Since \(r_1(t_2, \mu)=\varphi_1(t_2-t_1, \mu, \varphi_1(t_1, \mu, r_1^{\circ}))=r_1(t_1, \mu),\) a contradiction arises with the inequality in \eqref{MonotonicProof}.

Claim \ref{claim1}. Let \(r^{\star}_{1}\) be an equilibrium for the scalar equation \eqref{k-thEqPolarReduced} and the non-equilibrium solution \(r_1\) converge to \(r^{\star}_{1}.\) Therefore,  \(r^{\star}_{1}\tilde{f}(\mu, r^{\star}_{1})=0.\) So, there exists a maximum integer \(m\in \mathbb{N}\) such that the polynomial \(r_{1}\tilde{f}\) can be divided by \((r_1- r^{\star}_{1})^m\). Hence, there is a \(h\in \mathbb{R}[\mu, {r_1}]\) such that \(r_{1}\tilde{f}=(r_1- r^{\star}_{1})^mh\) and \(h(\mu, r^{\star}_{1})\neq 0\). By equation \eqref{k-thEqPolarReduced}, there exist real constants \(a_i\) for \(1\leq i\leq m\) and \(h^{\prime}\in \mathbb{R}[\mu, {r_1}]\) so that
\begin{eqnarray*}
&t=\int^{r_1(t, \mu)}_{r_1(0,\mu)}\!\frac{{\rm d}r_1}{(r_1\!-\!r^{\star}_{1})^mh}\!=\!\int^{r_1(t, \mu)}_{r_1(0,\mu)}\!\left(\sum_{i=1}^{m}\frac{a_i}{(r_1\!-\!r^{\star}_{1})^i}\!+\!\frac{h^{\prime}}{h}\right)\!{\rm d}r_1.
\eas
Integrating the later equation, we have
\bas
&\left(\sum_{i=1}^{m-1}\frac{-a_{i+1}(i)^{-1}}{(r_1(t, \mu)\!-\!r^{\star}_{1})^{i}}\!+a_1\ln|r_1(t, \mu)\!-\!r^{\star}_{1}|\right)+\left(\sum_{i=1}^{m-1}\frac{a_{i+1}(i)^{-1}}{(r_1^\circ\!-\!r^{\star}_{1})^{i}}\! -a_1\ln |r_1(0, \mu)\!-\!r^{\star}_{1}|+\!\int^{r_1(t, \mu)}_{r_1(0,\mu)}\!\frac{h^{\prime} {\rm d}r_1}{h}\right).&
\end{eqnarray*} Since the function \(h\) is distanced from zero when \(r_1(t, \mu)\) converges to \(r^{\star}_{1},\) the expression in the second big-parentheses is bounded. Our argument here is that \(r_1(t, \mu)\) is a monotonic function and there is no  any equilibrium point between \(r_1^\circ\) and \(r_1^{\star}.\) Therefore,  \(h(\mu, r_1)\neq 0\) for all \(r_1\) between \(r_1^\circ\) and \(r_1^{\star}.\) The limit of the expression in the first big-parentheses is unbounded. The positive or negative infinity depends on the sign of \(a_{m}\) and increasing/decreasing type of the function \(r_1(t, \mu)\). However, the convergence occurs in forward time. Hence, the right hand side (and time) must converge to positive infinity. This concludes the claim on the maximal forward-time interval of the existence.

Claim \ref{claim3}. Consider equation \eqref{ReducedEqu2}. Hence, the equilibria are obtained by solving \(r_1\tilde{f}(\mu, r_1)=0\) for \(r=r_1^{\star}.\) By Claim 3 in Lemma \ref{lem1}, \(\frac{r_1^{\star}}{c_1}\cb\) is an equilibrium for \eqref{ReducedEqu2}. Let \(r_1^{\star}\) be asymptotically stable for equation \eqref{k-thEqPolarReduced}. Hence, for every \(\epsilon>0,\) there is a \(\delta(\epsilon)>0\) such that the solution \(r_1(t, \mu)\) associated with every initial condition \(r_1(0, \mu)= r_1^{\circ},\) \(|r_1^{\circ}-r_1^{\star}|<\delta\),  satisfies \(|r_1(t, \mu)-r_1^{\star}|<\epsilon\) for all \(t\geq0.\) Moreover, there exists \(\eta(\epsilon)>0\) such that \(|r_1(t, \mu)-r_1^{\star}|\rightarrow 0\) as \(t\rightarrow \infty\) for all \(|r_1^{\circ}-r_1^{\star}|<\eta.\) Since
\begin{eqnarray*}
&| r_i(t, \mu) -r_i^{\star}|=|\frac{c_i}{c_1}r_1(t, \mu)-r_i^{\star}|= \frac{c_i}{c_1}|r_1(t, \mu)-\frac{c_1}{c_i}r_i^{\star}|=\frac{c_i}{c_1}|r_1(t, \mu)-r_1^{\star}|,&
\end{eqnarray*} for every \(\epsilon^{\prime}>0,\) we take
\begin{eqnarray*}
&\delta^{\prime}(\epsilon^{\prime}):= \sum_{i=0}^{n}\frac{c_i}{c_1}\delta\big(\frac{\epsilon^{\prime}}{\sum_{i=0}^{n}\frac{c_i}{c_1}}\big)\quad \hbox{ and } \quad \eta^{\prime}(\epsilon^{\prime}):=\sum_{i=0}^{n}\frac{c_i}{c_1}\eta\big(\frac{\epsilon^{\prime}}{\sum_{i=0}^{n}\frac{c_i}{c_1}}\big).&
\end{eqnarray*} The proof is now completed by
\bas
&\|\mathbf{r}(t, \mu)- \frac{r_1^{\star}}{c_1}{\pmb{c}}\|_2\leq \sum^n_{i=0}|r_i(t, \mu)- r_i^{\star}|= \sum^n_{i=0}\frac{c_i}{c_1}|r_1- r_1^{\star}|< \epsilon^{\prime}.&
\eas The converse claim is proved by a similar argument and taking
\begin{eqnarray*}
&\delta(\epsilon):= \frac{\delta^{\prime}(\epsilon\sum_{i=0}^{n}\frac{c_i}{c_1})}{\sum_{i=0}^{n}\frac{c_i}{c_1}}\quad\hbox{ and } \quad \eta(\epsilon):= \frac{\eta^{\prime}(\epsilon\sum_{i=0}^{n}\frac{c_i}{c_1})}{\sum_{i=0}^{n}\frac{c_i}{c_1}}.&
\end{eqnarray*} The unstable claim is simply the contraposition of the stable claim.

Claim \ref{claim4}. Let \(\Gamma\) be an invariant manifold and \((\x, \y)\in \Gamma\). By equations \eqref{x_i, y_i}, we have \(r_i(t, \mu)=\frac{r^\star_1 c_i}{c_1}\) for \(0\leq i\leq n,\) and \(\frac{\rm d}{{\rm d}t}\mathbf{r}=0\). Thus,
\(\frac{r^\star_1}{c_1}\cb\) and \(r^\star_1\) are the equilibria for \eqref{ReducedEqu2} and
\eqref{k-thEqPolarReduced}, respectively. For the \(\mathcal{M}_{\pmb{c}}\)-asymptotical stability, we first introduce the distance
\begin{eqnarray*}
&d( (\x, \y), \Gamma)\!:=\!\inf\limits_{\scriptsize (\mathbf{u}, \mathbf{v})\in\Gamma}\!d((\x, \y), (\mathbf{u}, \mathbf{v}))
\!:=\!\inf\limits_{\scriptsize (\mathbf{u}, \mathbf{v})\in\Gamma}\sum\limits_{i=0}^{n}\!\left((x_i\!-\!u_i)^2\!+\!(y_i\!-\!v_i)^2\right).&
\end{eqnarray*} Now we claim that
\begin{eqnarray}\label{distance-norm}
&d((\x(t, \mu), \y(t, \mu)), \Gamma)=\mylvert \mathbf{r}(t, \mu)-\frac{r_1^{\star}}{c_1}{\pmb{c}}\myrvert_2.&
\end{eqnarray} This is due to
\begin{eqnarray*}
&d( (\x, \y), \Gamma) = \sum_{i=0, i\neq n_j}^{n}\inf\limits_{\substack{a_i\in \mathbb{R}}} \left((r_i\cos(\omega_it +  \theta^{\circ}_i) - \frac{c_ir_1^{\star}}{c_1}\cos a_i)^2 + (r_i\sin(\omega_it  +  \theta^{\circ}_i)- \frac{c_ir_1^{\star}}{c_1}\sin a_i)^2\right)&\\
& = \sum_{i=0, i\neq n_j}^{n} \big((r_i- \frac{c_ir_1^{\star}}{c_1})^2\cos^2(\omega_it +  \theta^{\circ}_i) + (r_i- \frac{c_ir_1^{\star}}{c_1})^2\sin^2(\omega_it  +  \theta^{\circ}_i)\big)= \sum_{i=0}^{n}\left(r_i(t, \mu) - \frac{c_ir_1^{\star}}{c_1}\right)^2.&
\end{eqnarray*} Note that the interchange of the infimum and summation is justified due to the independence of the expressions with variations of the indices. The infimum here takes place when \(a_i= \omega_it +\theta^{\circ}_i\). Since \(\Gamma\) is \(\mathcal{M}_{\pmb{c}}\)-asymptotically stable (by definition) if and only if for every \(\epsilon>0\) there exists an open set \(N(\epsilon)\subset \mathbb{R}^{2n+2}\) such that \(\Gamma\subset N\) and for all \(((\x(0, \mu),\y(0, \mu))\in N(\epsilon)\cap \mathcal{M}_{\pmb{c}}\) and \(t\geq 0\), \(d((\x(t,\mu),\y(t,\mu)),\Gamma)< \epsilon\) and \(\lim_{t\rightarrow \infty} d((\x(t,\mu), \y(t,\mu)),\Gamma)=0\), the proof is straightforward.
\end{proof}

\begin{rem}
The maximal forward-time (or backward-time) of existence for the solutions of \eqref{EulerianDiff} is not necessarily an infinite interval. For example, let
\(f(\mu, \x, \y): =\alpha(\mu-{x_1}^2-{y_1}^2).\) Thus, the first amplitude equation in \eqref{PolarEq1} follows \(\frac{{\rm d}r_1}{{\rm d}t}=\alpha r_1(\mu-{r_1}^2).\) Let \(0<\mu<{r_1^{\circ}}^2\). Then, the solution and time interval of the existence are given by
\begin{eqnarray*}
&r_1(t, \mu)=\big(\frac{\mu {r_1^{\circ}}^2}{(\mu-{r_1^{\circ}}^2)\exp(\frac{2\mu t}{-\alpha})+{r_1^{\circ}}^2}\big)^{\frac{1}{2}},  \big[\frac{-\alpha}{2\mu}\ln\big(\frac{{r_1^{\circ}}^2}{{r_1^{\circ}}^2-\mu}\big), +\infty\!\big) \; \hbox{ for } \alpha>0,\; \big(\!-\infty,  \frac{\alpha}{2\mu}\ln\big(\frac{{r_1^{\circ}}^2}{{r_1^{\circ}}^2-\mu}\big) \big] \; \hbox{ for } \alpha<0.&
\end{eqnarray*}
When \(\mu>{r_1^{\circ}}^2,\) the time interval of the existence is the whole real line for every \(\alpha\in \mathbb{R}.\)
\end{rem}

Some trajectories of the system \eqref{EulerianDiff} are associated with the timber modeling of notes in section \ref{SecModel}. Our proposed modeling is made such that these trajectories converge to either an equilibrium or to an invariant hypertorus. Therefore, there is no forward time-interval limitation for the existence of simulated solutions. Furthermore, the proposed simulations are always bounded; see Claim \ref{claim1} in Lemma \ref{lem2} and the following corollary.

\begin{cor}\label{cor1} Consider the differential system \eqref{ReducedEqu2}. Then,
\begin{enumerate}
\item [i.] There is a one to one correspondence between positive roots of \(\tilde{f}(\mu, r^{\star}_1)=0\) and the flow-invariant \(n\)-tori with radiuses \(\frac{r_1^{\star}}{c_1}{\pmb{c}}\) on the leaf \(\mathcal{M}_{{\pmb{c}}}\) from system \eqref{EulerianDiff}.
\item[ii.] For a fixed parameter \(\mu,\) the trajectories \(r_i(t, \mu)\) for all \(i\leq n\) are either simultaneously increasing or simultaneously decreasing.
In addition, they either uniformly converge to a nonnegative real number or uniformly diverge to infinity.
\item[iii.] 
 When \((\x(t, \mu), \y(t, \mu))\) approaches either to an equilibrium or to an invariant hypertorus, the maximal forward-time interval of the existence for \((\x, \y)\) is the positive real line.
\end{enumerate}
\end{cor}
\begin{proof} Item (i) follows from claims \ref{monotonicity} and \ref{claim4} in Lemma \ref{lem2}. Item (ii) is trivial. When \((\x(t, \mu), \y(t, \mu))\) approaches to an invariant hypertorus, signals \(x_i(t, \mu)\) and \(y_i(t, \mu)\) are bounded. Thus, \(r_1(t)\) is bounded and monotonic by Claim \ref{Item1-1} in Lemma \ref{lem2}. Therefore, \(r_1\) converges to \(r_1^{\star}\) for some \(r_1^{\star}\in \mathbb{R}\) as the solutions converge to the hypertorus. By Claim \ref{Item1-1} in Lemma \ref{lem1}, there exists a \({\pmb{c}}\) such that \(\mathbf{r}\) converges to \(\frac{r_1^{\star}}{c_1}{\pmb{c}}.\) By equation \eqref{x_i, y_i}, \((\x(t, \mu), \y(t, \mu))\) converge to \(\Gamma\) given by equation \eqref{eq4.12}. By Claim \ref{claim4} in Lemma \ref{lem2}, \(\tilde{f}(\mu, \bar{r}_1)=0\) and \(\bar{r}_1\) is an equilibrium for \eqref{k-thEqPolarReduced}. Claim \ref{claim1} in Lemma \ref{lem2} completes the proof. When \((\x(t, \mu), \y(t, \mu))\) approaches to the origin, the argument is similar.
\end{proof}
\begin{thm}(Bifurcations)\label{Bifurcations} Consider the differential system  \eqref{EulerianDiff} and equation \eqref{k-thEqPolarReduced} where \(f:= \mu+g,\) \(g\in \mathbb{R}[{x_0}^2+{y_0}^2, \ldots, {x_n}^2+{y_n}^2]\) and \(f(\mu, \0, \0)=\mu.\) Then, the equation \eqref{k-thEqPolarReduced} is not structurally stable at \((\mu, r_1)=(0, 0).\) Indeed, equation \eqref{k-thEqPolarReduced} undergoes either a (non-standard) subcritical pitchfork or supercritical pitchfork bifurcation at the variety
\be
T_{SubP}:=\{\mu\,|\, \mu=0\} \quad ( T_{SupP}:=\{\mu\,|\, \mu=0\}),
\ee respectively. This corresponds with the appearance/disappearance of one \(n\)-hypertorus from the origin when the parameter \(\mu\) changes its sign in the differential system \eqref{EulerianDiff}.
\end{thm}
\begin{proof} We have \(D_{\x,\y}(\Theta^{\pmb{\omega}}_{\0}+E_f)(\mu, \0)=\diag(J_0, J_1, \ldots, J_n)\) where \({\scriptsize J_i:=\begin{pmatrix}\mu\!&\! -\omega_i\\ \omega_i\!&\!\mu \end{pmatrix}}\). Hence, the origin is an unstable equilibrium for \(\mu>0\) while the origin is locally asymptotically stable for \(\mu<0\). Given the hypothesis, the function \(\tilde{f}\) is read by \(\tilde{f}:=\mu+a{r_1}^{2p}+\hot\) when \(p\geq 1.\) Let \(a<0.\) The graph \(\tilde{f}\) passes through the origin within the \((r_1, \mu)\)-plane. Thereby, the function \(\mu(x)=-a{r_1}^{2p}+\mathcal{O}(2p+2)\) has a local maximum at \(r_1=0.\) Hence, for a sufficiently small neighbourhood of the origin when the parameter \(\mu\) varies from negative to positive, two nonzero asymptotically stable local equilibria are bifurcated from the origin. Thus, a (nonstandard) supercritical pitchfork type of bifurcation occurs for this case. The argument for the case \(a>0\) is similar and corresponds to a subcritical type case. The proof for system \eqref{EulerianDiff} is completed by appealing to Lemma \ref{lem2} and Corollary \ref{cor1}.
\end{proof}

\begin{thm}\label{lem3} Consider the differential equation
\ba\label{EqExm}
&\frac{\mathrm{d} r_1}{\mathrm{d}t}=r_1 \tilde{f}(\mu, r_1), \;\hbox{ where }\, \tilde{f}:=\alpha(\mu+a{r_1}^{2p}+b{r_1}^{2q}),\, \alpha\in\mathbb{R}_{>0}, \, p, q\in \mathbb{N},\, q>p,\, \hbox{ and } ab< 0.\qquad&
\ea When parameter \(\mu\) crosses the transition variety
\be T_{2SN}:=\left\{\mu\,\Big|\, \mu=\mu^{\star}=-a\left(\frac{-pa}{qb}\right)^{\frac{p}{q-p}}\left(\frac{q-p}{q}\right)\right\},\ee equation \eqref{k-thEqPolarReduced} exhibits a double (non-standard) saddle-node type of bifurcation. When \(a<0,\) there are two pairs of equilibria along with the origin when \(0<\mu< \mu^\star.\) These pairs coalesce and disappear when \(\mu\) crosses \(T_{2SN}\) and then \(\mu>\mu^\star.\) For \(a>0,\) there are four equilibria (apart from the origin) when \(\mu^\star< \mu<0\) and the origin is the only equilibrium for \(\mu<\mu^\star.\) When \(a<0.\), in terms of the differential system \eqref{EulerianDiff} reduced on \(\mathcal{M}_{\pmb{c}},\) two hypertori coalesce and disappear when the parameter \(\mu\) varies from
\(0<\mu< \mu^\star\)  to \(\mu>\mu^\star\).

\end{thm}
\begin{proof}
Take \(\mu(r_1)=-a{r_1}^{2p}-b{r_1}^{2q},\) \(r_1^{\star, \pm}:=\pm\left(\frac{-pa}{qb}\right)^{\frac{1}{2(q-p)}}\) and \(\mu(r_1^{\star, \pm})=-a\left(\frac{-pa}{qb}\right)^{\frac{p}{q-p}}\left(\frac{q-p}{q}\right).\) We have
\begin{eqnarray*}
&\frac{\mathrm{d}}{\mathrm{d}r_1}\mu\!\left(r_1^{\star, \pm}\right)=0, \quad\frac{\mathrm{d}^2}{\mathrm{d}{r_1}^2}\mu\!\left(r_1^{\star, \pm}\right)=4a(q-p)\left(\frac{-pa}{qb}\right)^{\frac{p-1}{q-p}}\neq 0.&
\end{eqnarray*}
Hence, \(\tilde{f}\) has a global minimum (maximum) at \(r_1=r_1^{\star, \pm}\) for \(a<0\) \((a>0).\) Therefore, two saddle node type of bifurcations take place simultaneously at \(\mu=\mu^\star\); see Figure \ref{BifDiag}. Theorem \ref{Bifurcations} concludes the claim on the system \eqref{EulerianDiff} restricted on \(\mathcal{M}_{\pmb{c}}.\)
\end{proof}
\begin{exm} Consider the differential equation \eqref{EqExm} for \(\alpha:=23,p:=1, q:=2, |a|:=1,\) and \(|b|:=2.15.\) These constants correspond with the values used in section \ref{SecModel} for modeling \(C\sharp 4\) played by violin. 
  When \((a, b):=(1, -2.15),\) a supercritical pitchfork bifurcation occurs at \(T_{SupP}\{\mu| {\mu}=0\}\) and for \((a, b):=(-1, 2.15),\) a subcritical pitchfork takes place at \(T_{SubP}\{\mu| {\mu}=0\}\); see Figure \ref{SubP2SN} and \ref{SupP2SN}. A double saddle-node bifurcation occurs at \(T_{2SN}= \{\mu| \mu=\mu^\star\}.\) This is associated with bifurcation points \((\mu^\star, r_1^{\star, \pm})=(-0.116, \pm 0.5).\) Now consider \(b=0\) and \(a=\pm 1.\) In this case, there is a supercritical pitchfork bifurcation at \(\mu=0\) when \(a=-1\) and a subcritical pitchfork bifurcation when \(a= 1\); see Figures \ref{SupSubP}.
\end{exm}
\begin{figure}[t]
\centering
\subfloat[Here, \(a:=1, b:=-2.15\). A hysteresis cycle occurs by increasing and decreasing \(\mu.\) \label{SubP2SN}]{\includegraphics[width=2.3in, height=2.3in]{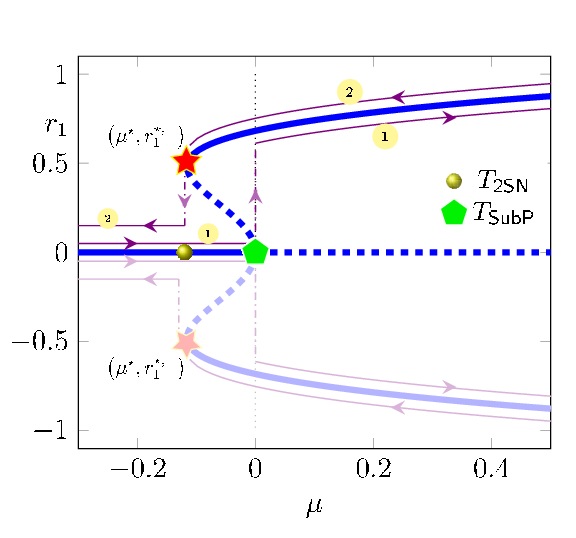}}\;
\subfloat[The case \(a:=-1, b:=2.15\) and phase portraits for fixed values of \(\mu.\) \label{SupP2SN}]{\includegraphics[width=2.2in,height=2.3in]{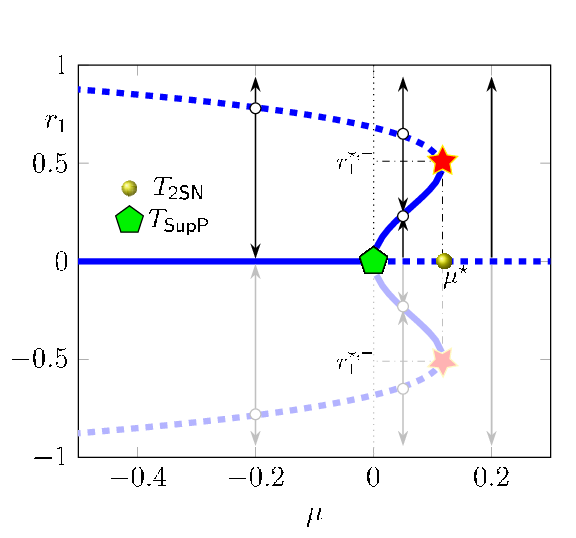}}\;
\subfloat[The supercritical and subcritical pitchfork bifurcations for \(a:=\pm 1, \) and \(b:=0,\) respectively. \label{SupSubP}]
{\includegraphics[width=2.24in,height=2.2in]{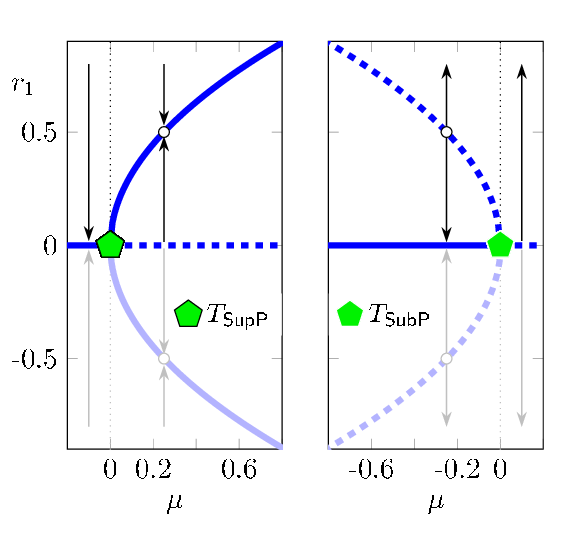}}
\caption{Bifurcation diagrams associated with equation \eqref{EqExm}. }\label{BifDiag}
\end{figure}

\begin{rem}[{\it Hysteresis cycle}: quasi-static variation of \(\mu\) and convergence switching] Consider equation \eqref{EqExm} when \(a>0\) and \(b<0\); for an instance see Figure \ref{SubP2SN}. We use a quasi-static variation (piecewise-constant) of the parameter \(\mu\) and observe the asymptotic behavior of \eqref{EqExm}.
We initially assume that \(0<r_1^{\circ}<r_1^{\star}\) and \(\mu<\mu^\star<0.\) Then, the differential equation has five equilibria: the origin, two positive and two negative equilibria. The origin is asymptotically stable and the initial condition guarantees the convergence to the origin. Now we consider to quasi-statically increase the parameter \(\mu\); see curve \tikz[baseline=(char.base)]{\node[shape=circle,fill=yellow!50, scale=0.8,draw=yellow!50,inner sep=2pt] (char) {1};} in Figure \ref{SubP2SN}. As long as the parameter \(\mu\) is negative, the trajectories converge to the origin. However, as soon as the parameter \(\mu\) changes its sign to positive, the origin becomes unstable and two equilibria (one positive and one negative) coalesce with the origin and disappear. Two other equilibria are distanced from the origin and becomes asymptotically stable. The positive equilibrium is larger than \(r_1^{\star}\). Since the initial condition is positive, the convergence of the trajectory switches to the positive equilibrium.
The positive asymptotically stable equilibrium continues to exists when the parameter \(\mu\) increases any further. The controller parameter \(\mu\) manages the amplitude sizes of the simulated sound. Here, the simulated amplitudes of the harmonic partials and the radiuses of the corresponding flow-invariant hypertori in the differential system increase (decrease) when we amplify (reduce) the parameter \(\mu\) within the positive real line.

Now consider to decrease the parameter \(\mu\) while  \(r_1^{\circ}>r_1^{\star}\); see curve \tikz[baseline=(char.base)]{\node[shape=circle,fill=yellow!50, scale=0.8,draw=yellow!50,inner sep=2pt] (char) {2};} in Figure \ref{SubP2SN}. When the parameter \(\mu\) changes its sign and \(\mu^\star<\mu<0\), there are three asymptotically stable equilibria (a zero, a negative and a positive) and two unstable equilibria (one is positive and one is negative) for \eqref{EqExm}. However, the initial condition is above the unstable positive equilibrium and thus, the trajectory continues to converge to the positive asymptotically stable equilibrium. When the parameter \(\mu\) further decreases so that \(\mu<\mu^\star,\) all non-zero equilibria vanish through a double saddle-node bifurcation and the differential equation only has an (asymptotically stable) equilibrium at the origin. Therefore, the convergence of solutions switches back to the origin and completes the hysteresis cycle. These consecutive changes of stabilities, the jumping up and falling down of the asymptotically stable equilibria cause a hysteresis dynamics for the trajectories of the differential equation. By item (i) in Corollary \ref{cor1}, the hysteresis cycle corresponding with system \eqref{EulerianDiff} reduced on leaf \(\mathcal{M}_{{\pmb{c}}}\) involves with bifurcations of asymptotically stable hypertori. These type of hysteresis cycles occur in the dynamics of the musical notes when there are two steady-states in the sustain level of the note. For an instance, we observe the hysteresis in Section \ref{SecModel} when we deal with \(C\sharp 4\) played by violin.
\end{rem}

\section{Estimated peaks of the amplitude spectrum }\label{Sec4}

Every musical note contains four properties; namely, pitch, duration, loudness  and timbre. {\it Spectral} and {\it temporal} envelops are two of the main characteristics of timbre. Temporal envelop demonstrates the changes of note's amplitude over time. Spectral envelop is a curve in the frequency-amplitude plane, where it smoothly interpolates the partial peaks of the fourier transform of the audio signal. The estimation of the individual partial peaks and their ratio to the amplitude of the fundamental frequency are sufficient for an accurate approximation of the spectral envelop; \eg see \cite{SensationsOfTone,MathematicsAndMusic}. Therefore, an accurate simulation of timber via spectral and temporal envelops collectively addresses all four properties of a note.

\begin{lem}\label{lem4}
Consider the differential system \eqref{EulerianDiff}, \(f\in \mathbb{R}[\mu,{x_0}^2+{y_0}^2, \ldots, {x_n}^2+{y_n}^2],\) and \(\mathcal{F}(r_k)(\omega, \mu)\) as the Fourier transform of the amplitude signal \(r_k\) for \(0\leq k\leq n\) over \(0\leq t\leq \tau.\) Then,
 \begin{enumerate}
   \item\label{411Item1}  \(|\omega \mathcal{F}(r_k)(\omega, \mu)|< 2\sqrt{2}\max\{r_k^{\circ}, r_k(\tau, \mu)\}.\)
   \item\label{411Item2} The unique global maximum (peak) point of the amplitude spectrum for \(r_k\) is \(\mathcal{F}(r_k)(0, \mu)\).
 \end{enumerate}
\end{lem}
\begin{proof}
Claim \ref{411Item1}. We have
\begin{eqnarray}\label{Rkomega}
&\mathcal{F}(r_k)(\omega, \mu)=\int_{0}^{\tau}r_k(t, \mu)\exp(-\mathbf{i}\omega t)\,\mathrm{d}t=\int_{0}^{\tau}r_k(t, \mu) \cos\omega t\,\mathrm{d}t-\mathbf{i}\int_{0}^{\tau}r_k(t, \mu)\sin\omega t\,\mathrm{d}t.&
\end{eqnarray}
Since \(r_k\) is the \(k+1\)-th component of a solution of the differential equation \eqref{PolarEq1}, \(r_k\) is non-negative and a strictly monotonic function for fixed values of \(\mu\). Now assume that \(r_k(t, \mu)\) is an increasing function. By \emph{Bonnet Theorem} (the second mean value theorem) there exist \(t_1, t_2\in [0, \tau]\) such that
\begin{eqnarray*}
&\int_{0}^{\tau}r_k(t, \mu) \cos\omega t\, \mathrm{d}t= r_k(\tau, \mu) \int_{t_1}^{\tau}\cos\omega t\, \mathrm{d}t, \quad \int_{0}^{\tau}r_k(t, \mu) \sin\omega t\, \mathrm{d}t= r_k(\tau, \mu)\int_{t_2}^{\tau}\sin\omega t\, \mathrm{d}t.&
\end{eqnarray*} 
The above relations yield
\begin{eqnarray*}
&\mathcal{F}(r_k)(\omega, \mu)=r_k(\tau, \mu)(\tau\sinc\,(\omega\tau)-t_1\sinc\,(\omega t_1)+\mathbf{i}\,\tau\,\sin(\frac{\omega\tau}{2})\sinc(\frac{\omega\tau}{2})-\mathbf{i}\,t_2\,\sin(\frac{\omega t_2}{2})\sinc(\frac{\omega t_2}{2})),&
\end{eqnarray*} where the function \(\sinc\,\alpha\) is defined by \(\frac{\sin\alpha}{\alpha}\) for \(\alpha\ne 0\) and is \(1\) for \(\alpha=0\). Let \(\omega \neq 0.\) We have
\begin{eqnarray}\label{ineq1}
&\omega^2|\mathcal{F}(r_k)(\omega, \mu)|^2={r_k(\tau, \mu)}^2((\sin\omega\tau-\sin\omega t_1)^2+(\cos\omega\tau-\cos\omega t_2)^2)<8{r_k(\tau, \mu)}^2.&
\end{eqnarray}
When \(r_k(t, \mu)\) is a descending function, by a similar argument we have
\begin{eqnarray*}
&\mathcal{F}(r_k)(\omega, \mu)=r_k(0, \mu)\left(t_3\,\sinc(\omega t_3)+\mathbf{i}\,t_4\,\sin(\frac{\omega t_4}{2})\sinc(\frac{\omega t_4}{2})\right)&
\end{eqnarray*} and \(|\omega \mathcal{F}(r_k)(\omega, \mu)|< \sqrt{5}r_k^{\circ}\) for \(\omega \neq 0.\) This completes the proof.

Claim \ref{411Item2}. From \(\mathcal{F}(r_k)(\omega, \mu)=\int_{0}^{\tau}r_k(t, \mu)\exp(-\mathbf{i}\omega t)\mathrm{d}t,\) we have
\begin{eqnarray*}
|\mathcal{F}(r_k)(\omega, \mu)|\leq \int_{0}^{\tau}| r_k(t, \mu)\exp(-\mathbf{i}\omega t)|\mathrm{d}t=\int_{0}^{\tau}r_k(t, \mu)\mathrm{d}t=\mathcal{F}(r_k)(0, \mu).
\end{eqnarray*} Hence, \(\mathcal{F}(r_k)(0, \mu)\) is a maximum value for \(|\mathcal{F}(r_k)(\omega, \mu)|.\) By equation \eqref{Rkomega}, we have
\begin{eqnarray*}
|\mathcal{F}(r_k)(\omega, \mu)|^2= \int_{0}^{\tau}\int_{0}^{{\tau}}r_k(t, \mu)r_k(t^{\prime}, \mu)\cos\omega(t-t^{\prime})\mathrm{d}t^{\prime}\mathrm{d}t.
\end{eqnarray*} Therefore,
\begin{eqnarray*}
&|\mathcal{F}(r_k)(\omega, \mu)|^2-|\mathcal{F}(r_k)(0, \mu)|^2=\textstyle\int_{0}^{\tau}\int_{0}^{{\tau}}r_k(t, \mu)r_k(t^{\prime}, \mu)(\cos\omega(t-t^{\prime})-1)\mathrm{d}t^{\prime}\mathrm{d}t\qquad\qquad\qquad\qquad\qquad& \\
&\leq \textstyle {\min\{r_k^{\circ}, r_k(\tau, \mu)\}}^2\int_{0}^{\tau}\int_{0}^{{\tau}}(\cos\omega(t-t^{\prime})-1)\mathrm{d}t^{\prime}\mathrm{d}t \qquad\qquad\qquad\qquad\qquad\qquad&\\
&=\textstyle {\min\{r_k^{\circ}, r_k(\tau, \mu)\}}^2\left(\frac{2(1-\cos\omega\tau)}{\omega^2}-\tau^2\right)=\frac{4 {\min\{r_k^{\circ}, r_k(\tau, \mu)\}}^2}{\omega^2}\left(\sin^2\left(\frac{\omega\tau}{2}\right)-\left(\frac{\omega\tau}{2}\right)^2\right).&
\end{eqnarray*} Since \(\sin \alpha< \alpha\) for each nonzero \(\alpha\in \mathbb{R},\) the latest expression is always strictly negative for \(\tau\neq 0\). Therefore, the proof of the second claim is now completed.
\end{proof}

\begin{cor}\label{cor2} Let \(n:=6,\) \(\omega_k=k\omega_1\) for \(1\leq k\leq n,\) \((\x(t, \mu), \y(t, \mu))\) be a solution for the differential system \eqref{EulerianDiff}, where \((\x, \y)\in \mathcal{M}_{{\pmb{c}}}\) over the time interval \(0\leq t\leq \tau.\) Further, denote \(X_i(\omega, \mu)\) and \(Y_i(\omega, \mu)\) (\(1\leq i\leq n\)) for the Fourier transform of \(x_i(t, \mu)\) and \(y_i(t, \mu)\). Then, for \(X(\omega, \mu):= \sum^n_{i=1}X_{i}(\omega, \mu)\) and all \(k,\)
\bas
&\vert X(\omega, \mu)-\frac{1}{2} \mathcal{F}(r_k )(\omega-\omega_k, \mu)\vert < \frac{7}{127} \qquad \hbox{ when }\, \omega\in \left(\frac{(2k-1)\omega_1}{2}, \frac{(2k+1)\omega_1}{2}\right).&
\eas Furthermore, estimated peaks of the amplitude spectrum of the signal \(\sum_{i=1}^{n}x_i\) follow
\begin{eqnarray*}
&\left(\frac{\omega_i}{2\pi}, \frac{c_i}{2\,c_1}\mathcal{F}(r_1)(0, \mu)\right) \qquad\hbox{ for }\;  1\leq i\leq n.&
\end{eqnarray*}
\end{cor}
\begin{proof} From equations
\begin{eqnarray*}
&X_i\!\pm\!\mathbf{i}Y_i\!=\!\int_{0}^{\tau}\!r_i(\cos\omega_it\pm\mathbf{i}\sin\omega_it)\exp(-\mathbf{i}\omega t)\mathrm{d}t\!=\!
\int_{0}^{\tau}\!r_i\exp(-\mathbf{i}(\omega\!\mp\!\omega_i)t)\mathrm{d}t\!=\!\mathcal{F}(r_i)(\omega\!\mp\!\omega_i, \mu), &
\end{eqnarray*} we obtain the modulation-equation given by
\begin{eqnarray}\label{ModulationTheorem}
  &X_i(\omega, \mu)=\frac{1}{2}\mathcal{F}(r_i)(\omega-\omega_i, \mu)+\frac{1}{2}\mathcal{F}(r_i)(\omega+\omega_i, \mu).&
\end{eqnarray} The idea is to prove that the left side bound \(\mathcal{F}(r_i)(\omega+\omega_i, \mu)\) can be ignored and the main contributor into the sound hearing quality is the right side bound \(\mathcal{F}(r_i)(\omega-\omega_i, \mu).\) Due to equation \eqref{Rkomega}, we have \(|\mathcal{F}(r_i)(\omega, \mu)|=|\mathcal{F}(r_i)(-\omega, \mu)|,\) and thus, \(|X_i(\omega, \mu)|=|X_i(-\omega, \mu)|.\) Therefore, the function \(|X_i(\omega, \mu)|\) is symmetric with respect to the axis \(\omega=0\) and thus, we assume that \(\omega>0.\) Due to the strictly monotonic property of \(r_i(t, \mu)\) (by Corollary \ref{cor1}), we consider the increasing case for \(r_i(t, \mu)\). When \(r_i(t, \mu)\) is decreasing, the argument is similar. Hence, \(\max\{r_i^{\circ}, r_i(\tau, \mu)\}=r_i(\tau, \mu).\) We have
\begin{eqnarray*}
& X(\omega, \mu)=\int_{0}^{\tau}\sum_{i=1}^{n}x_i(t, \mu) \exp(-\mathbf{i}\omega)\mathrm{d}t=\sum_{i=1}^{n}\int_{0}^{\tau}x_i(t, \mu)\exp(-\mathbf{i}\omega)\mathrm{d}t=\sum_{i=1}^{n}X_i(\omega, \mu).&
\end{eqnarray*}
Let \(\omega_1\geq 27.5\)\,Hz. From equation \eqref{ModulationTheorem}, inequality \eqref{ineq1} and item \ref{411Item1} from Lemma \ref{lem2}, for the frequency interval \(\frac{(2k-1)\omega_1}{2}\leq\omega\leq \frac{(2k+1)\omega_1}{2}\) we have
\begin{eqnarray}\nonumber
&\vert X(\omega, \mu)-\frac{1}{2}\mathcal{F}(r_k)(\omega-\omega_k, \mu)\vert=\frac{1}{2}|\sum_{i=1}^{n}\mathcal{F}(r_i)(\omega+\omega_i, \mu)+\sum_{i=1,i\neq k}^{n}\mathcal{F}(r_i)(\omega-\omega_i, \mu)|&\\\nonumber
&\leq \frac{1}{2}\sum_{i=1}^{n}|\mathcal{F}(r_i)(\omega+\omega_i, \mu)|+\frac{1}{2}\sum_{i=1,i\neq k}^{n}|\mathcal{F}(r_i)(\omega-\omega_i, \mu)|&\\\nonumber
&\leq  \frac{\sqrt{2}\, r_k(\tau, \mu)}{\omega_1}\left(\sum_{i=1}^{n}\frac{1}{i+k-0.5}+\sum_{i=1}^{k-1}\frac{1}{i-0.5}+\sum_{i=1}^{n-k}\frac{1}{i-0.5}\right)&\\&
\leq \frac{\sqrt{2}\, r_k(\tau, \mu)}{\omega_1}\max_{1\leq k\leq n} \left\{\sum_{i=1}^{n}\frac{1}{i+k-0.5}+\sum_{i=1}^{k-1}\frac{1}{i-0.5}+\sum_{i=1}^{n-k}\frac{1}{i-0.5}\right\} <\frac{6.8\sqrt{2}\, r_k(\tau, \mu)}{55\pi}< \frac{7}{127}.&
\end{eqnarray}
The accuracy of these estimates further increases as the fundamental frequencies of notes amplifies. Thus, our approach works well for all note frequencies higher than that of note \(A0,\) that is, 27.5Hz. Thus, \(|X(\omega, \mu)-\frac{1}{2}\mathcal{F}(r_k)(\omega-\omega_k, \mu)|\) has a very small value and can be neglected. This gives rise to the accurate peak estimates for the amplitude spectrum at \((\frac{\omega_i}{2\pi}, \frac{1}{2}\mathcal{F}(r_i)(0, \mu)).\)
\end{proof}


\begin{figure}[t]
\centering
\subfloat[Piano waveform of pitch $C\sharp 4$, dynamics mark {\it {p}}\label{Csharp4P}]{\includegraphics[width=3.6in]{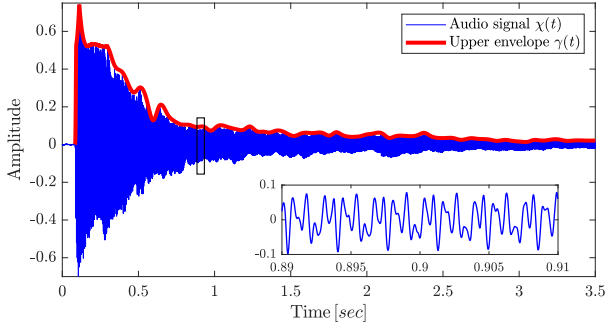}}
\subfloat[Violin waveform of pitch $C\sharp 4$, dynamics mark  {\it forte }\label{Csharp4V}]{\includegraphics[width=3.6in]{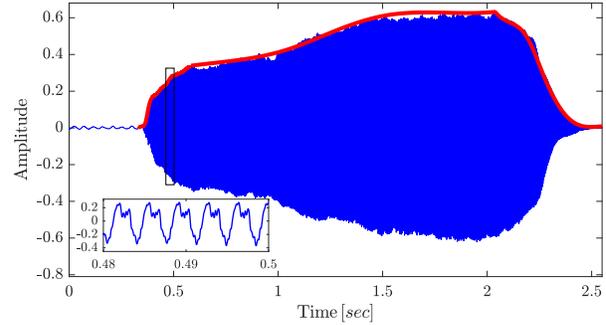}}\;
\caption{{\it Temporal envelops} for $C\sharp 4$ played by a piano and a violin}\label{Csharp4Plot}
\end{figure}

Let \(p_i=(\nu_i, d_i)\) denote for the \(i\)-th peak of the magnitude spectrum \(|\mathcal{F}(\chi(t))|\) where \(\nu_i\) and \(d_i\) are peak frequency values and peak amplitudes for \(1\leq i\leq n.\) Further,
\begin{eqnarray}\label{d0}
&\nu_0:=0,\quad d_0:=-\frac{1}{2}\left(\sum_{i=1}^{n}d_i+\min\left\{\sum_{i=1}^{n}d_i\cos(2\pi\nu_it)|\, 0\leq t\leq \nu_1^{-1}\right\}\right).&
\end{eqnarray}
The amplitude spectral vector \(\cb\) for each audio signal of a musical note is denoted by
\ba
&\cb:=(d_i)^n_{i=0}.&
\ea The relative sizes of the amplitudes \(d_i\)-s determine the timbre and these make a primary distinction between different musical sounds. The Eulerian structural symmetry essentially preserves the relative sizes of \(d_i\) from the initial conditions living on the leaf manifold. The signal \(\sum_{i=0}^{n}x_{i}\) represents the actual proposed simulated musical note, \ie this is what is actually heard, while the state variables of the differential equations are \((x_i, y_i)\) for \(i=0, 1, \dots, n\). The key argument in this decomposition and our Eulerian type modeling of a musical note lies in the {\it superposition} principle. In addition, human hearing is insensitive to the phase of harmonic partials; see \cite{DigitalSignalProcessing,TuningTimbreSpectrumScale}. Hence, the signals \(\sum_{i=1}^{n}(\delta_{i}x_{i}+(1-\delta_{i})y_{i})\) for different choices of \(\delta_{i}\in \{0, 1\}\) give rise to the same sound.

\subsection{Spectrum analysis for $C\sharp 4$ played by piano and violin}\label{SecSpectrm}

Consider the spectrum of the musical note $C\sharp 4$ played by piano and violin depicted in Figure \ref{D4-Csharp4Spectrum}. We use \texttt{fft()} function (fast fourier transform) in {\sc Matlab} to obtain these figures. Now consider the audio signal obtained from Piano given in Figure \ref{Csharp4P}. Following Remark \ref{Rem4.12}, we let \(n=6.\) The intensity of partials for \(i\geq 7\) is lower than the hearing threshold. Thus,
\begin{eqnarray*}
&d_1:=\frac{1069}{10000}, d_2:= \frac{923}{10000}, d_3:=\frac{604}{10000}, d_4:=\frac{411}{10000}, d_5:=\frac{559}{10000}, d_6:= \frac{412}{10000},&\\
&\nu_0=0, \nu_1:=274.4, \nu_2:= 548.9, \nu_3:= 823.3, \nu_4:= 1100, \nu_5:= 1376.6, \nu_6:= 1655.5&
\end{eqnarray*} and \(d_0=-0.1451.\) Hence, the harmonic frequency vector \(\pmb{\omega}_{\textrm{Piano}}^{C\sharp 4}\) and amplitude spectral vector \({\pmb{c}}_{\textrm{Piano}}^{C\sharp 4}\) corresponding with audio signal given in Figure \ref{Csharp4P} are computed as
\bes\pmb{\omega}_{\textrm{Piano}}^{C\sharp 4}:=2\pi (\nu_0, \nu_1, \ldots, \nu_{6}) \quad \hbox{ and } \quad {\pmb{c}}_{\textrm{Piano}}^{C\sharp 4}:=(d_0, d_1, \ldots, d_{6}).\ees
Using a similar procedure and argument for the case of Figure \ref{Csharp4V}, an audio signal of the pitch $C\sharp 4$ played with violin, we have
\begin{eqnarray*}
&&\pmb{\omega}_{\text{Violin}}^{C\sharp 4}:= 2\pi\left(0, 277.6, 555.2, 832.8,1110,1387.6, 1665.2\right),\\
&& {\pmb{c}}_{\text {Violin}}^{C\sharp 4}:={10^{-4}}\left(-1438, 3746, 1356, 421, 192, 119, 309 \right).
\end{eqnarray*}
When the amplitude spectral vector \({\pmb{c}}\) is specified, the flow-invariant manifold \(\mathcal{M}_{{\pmb{c}}}\) is determined.


\begin{figure}[t]
\centering
\subfloat[Normalised magnitude for $C\sharp 4$ played by piano \label{FourierPiano} ]{\includegraphics[width=3.6in]{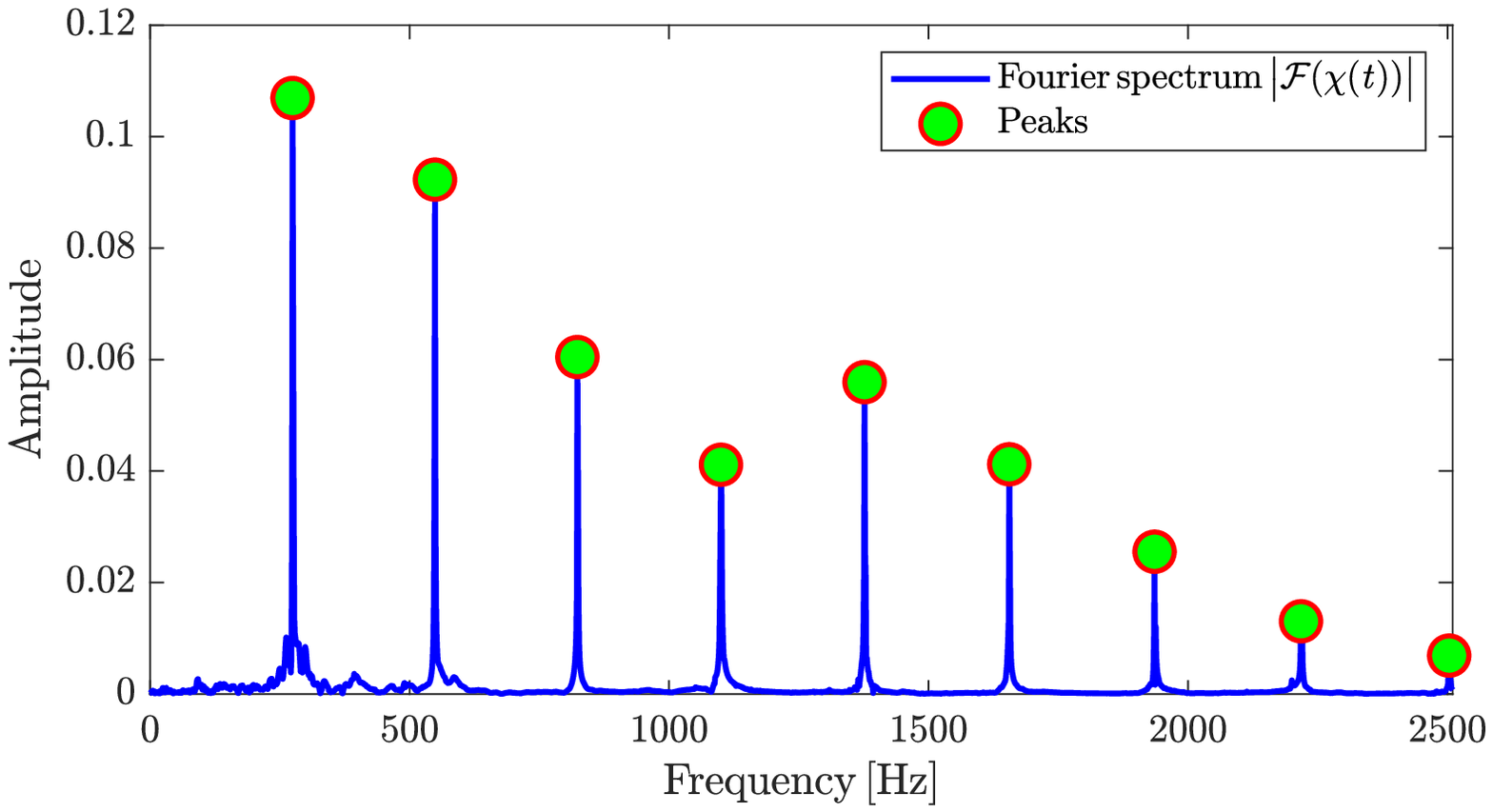}} 
\subfloat[The case pitch $C\sharp 4$ played by a violin \label{Fourierviolin}]{\includegraphics[width=3.6in]{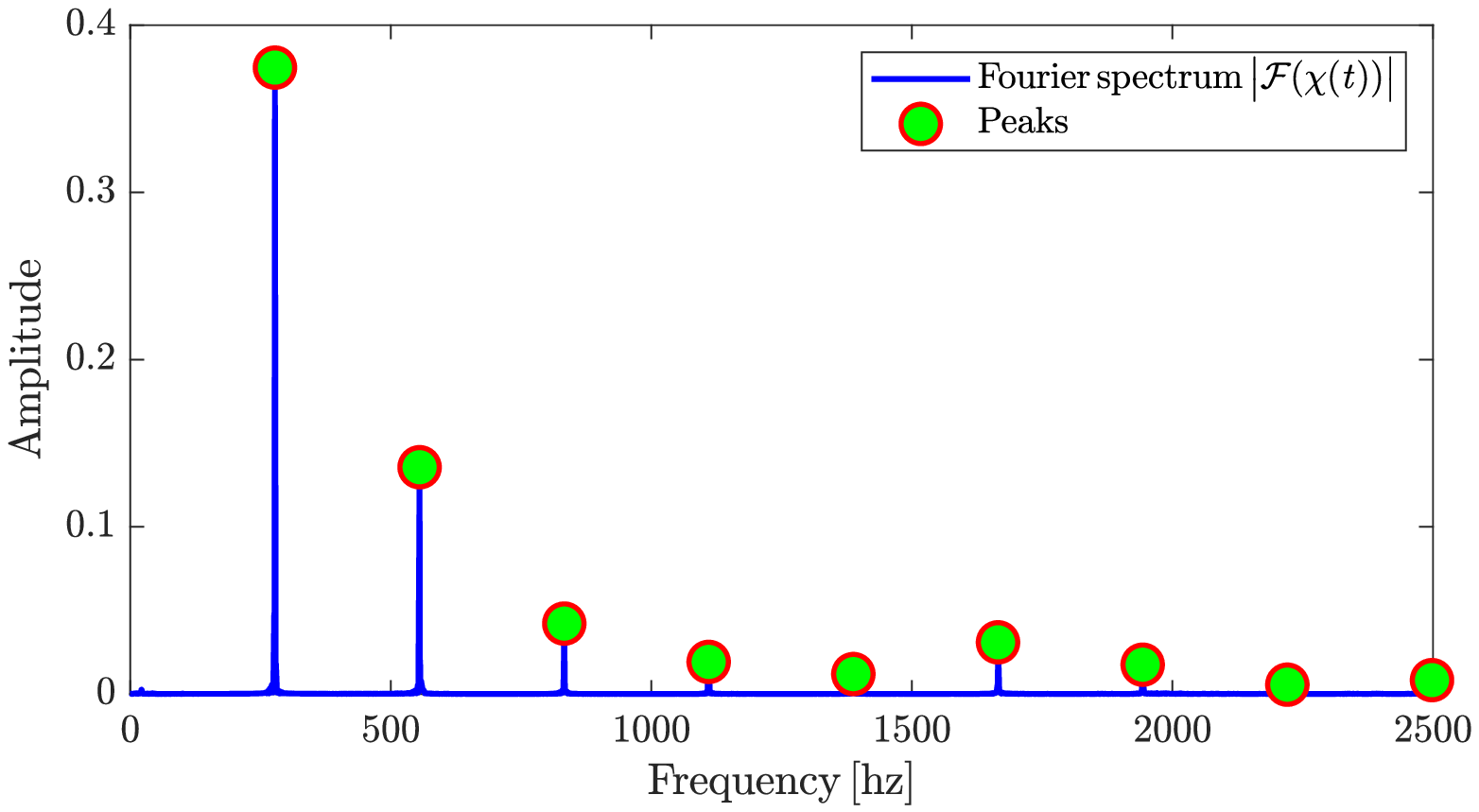}}\;
\caption{{\it Normalised magnitude spectrum} of the pitch $C\sharp 4$ played by piano and violin. }\label{D4-Csharp4Spectrum}
\end{figure}

\section{Dynamical system modelling of a musical sound}\label{SecModel}

The Eulerian structure preserves the spectral envelop simulated by its partial peaks. Thus, we only need to simulate the temporal envelop of a note and accommodate it into our Eulerian differential system modeling. Assume that \(\chi(t)\) is the signal of a musical sound with duration of \(t_f-t_0\). Denote \(\gamma(t)\) for the temporal envelope during the interval \([t_0, t_f]\). Individual segments of a temporal envelope such as dynamic states (attack, decay, release) and steady-states (delay, sustain and hold) are assumed to be piecewise monotonic functions. Hence, we split duration into segment subintervals denoted by \([t_j, t_{j+1}]\) for \(j=0, \ldots,m-1\) where \(t_{m}=t_f\) and \(m\) corresponds with the number of monotonic parts of the temporal envelope.

\begin{lem}\label{Lem5.1} Consider the differential equation \eqref{EulerianDiff} over the time interval \([t_j, t_{j+1}]\) along with the initial conditions \((x_i(t_j, \mu_j), y_i(t_j, \mu_j))=(\frac{d_i}{d_1}x_1(t_j, \mu_j), 0)\) and \(\omega_i:=2\pi\nu_i\) for \(0\leq i\leq n,\) where
\begin{eqnarray}\label{con1}
&f:=\alpha(\mu_j+g), \quad\hbox{ and }\quad g\in \mathbb{R}[{x_0}^2+{y_0}^2, \ldots, {x_n}^2+{y_n}^2].&
\end{eqnarray}
We assume that for a sufficiently small \(\epsilon>0\), the function \(g(\x, \mathbf{y})\) and the parameter \(\mu_j\) guarantee
\begin{eqnarray}\label{con2}
&\left\vert \sum_{i=0}^{n}\frac{d_i}{d_1} r_1(t, \mu_j)-\gamma(t) \right\vert<\epsilon, \quad \hbox{ for } \; t\in [t_j, t_{j+1}].&
\end{eqnarray} (Remark that the function \(g\) and \(\mu_j\) are appropriately introduced in equation \eqref{gForm} and Algorithm \ref{Algorthm} below.) Then, the relative sizes of partials to the fundamental amplitude associated with the signal \(\sum_{i=0}^{n}x_{i}\) are accurate estimations for those corresponding with \(\chi\). Furthermore, these two signals have effectively the same temporal envelop when \(\epsilon<\frac{7}{127}\).
\end{lem}
\begin{proof}
Let \(\theta_{i}^{\circ}=0\) for \(0\leq i\leq n.\) Due to the flow-invariant property of the manifold \(\mathcal{M}_{{\pmb{c}}},\)
we have
\begin{eqnarray*}
&\sum_{i=0}^{n}x_{i}=\sum_{i=0}^{n}r_{i}\cos(2\pi\nu_it)=\frac{(\sum_{j=0}^{n}d_j)(\sum_{i=0}^{n}\frac{d_ir_{1}}{d_1}\cos(2\pi\nu_it))}{\sum_{j=0}^{n}d_j}
=\frac{\left(\sum_{j=0}^{n}\frac{d_jr_{1}}{d_1}\right)(\sum_{i=0}^{n}d_i\cos(2\pi\nu_it)}{\sum_{j=0}^{n}d_j}.&
\end{eqnarray*} The function \(\sum_{i=0}^{n}d_i\cos(2\pi\nu_it)\) is an alternating signal with frequency \(\nu_1\) hertz. Since
\begin{eqnarray*}
&\sum_{i=0}^{n}d_i\cos(2\pi\nu_it)\geq \min\{\sum_{i=1}^{n}d_i\cos(2\pi\nu_it)|\, 0\leq t\leq \nu_1^{-1}\}+d_0=-\sum_{i=0}^{n}d_i,&
\\&-1\leq \sum_{i=0}^{n}\frac{d_i}{\sum_{j=0}^{n}d_j}\cos(2\pi\nu_it)\leq 1.&\eas
Here, the right equality holds when \(t=\frac{l}{\nu_1}\) for some \(l\in \mathds{Z}\). Hence, \(\sum_{j=0}^{n}\frac{d_j}{d_1}r_1\) is the temporal envelope of the signal \(\sum_{i=0}^{n}x_{i}\). The proof is completed by Corollary \ref{cor2} and equation \eqref{con2}.
\end{proof}

Now we extend Lemma \ref{Lem5.1} from a time segment to the whole time interval \([t_0, t_m].\)

\begin{cor}\label{cor3}
Consider the dynamical system
\begin{eqnarray}\label{SimDyn}
&\frac{\rm d}{{\rm d}t}(\x, \mathbf{y})=\Theta^{\pmb{\omega}}_{\0}+E_{\alpha(\mu(t)+g)}, \quad (x_i(0, \mu_0), y_i(0, \mu_0))=\frac{d_i}{d_1}(x_1(0, \mu_0), 0) \quad\text{for}\quad 0\leq i\leq n, &
\end{eqnarray} where \(\omega\) and \(g(\x, \mathbf{y})\) are given as of Lemma \ref{Lem5.1} and
\begin{eqnarray*}
&\mu(t):=\sum_{j=0}^{m}(\mu_j-\mu_{j-1})H(t-t_j). &
\end{eqnarray*} Here, \(H(t)\) is the Heaviside step function and \(\mu_{-1}=0\). Further, assume that the condition \eqref{con2} is satisfied for the entire interval \([t_0, t_{m}]\). Then, the dynamical system \eqref{SimDyn} provides an accurate estimation for the spectral and temporal envelops of \(\chi(t)\).
\end{cor}

Corollary \ref{cor3} concludes that the qualitative type changes in the dynamical system \eqref{SimDyn} represents the qualitative behavior of the audio signal of the musical note. Hence, for modelling of an arbitrary audio signal of a musical note, say \(\chi(t)\), we identify vectors \(\pmb{\omega}\) and \({\pmb{c}}\) via the Fourier spectrum analysis of the audio signal, \ie \(|\mathcal{F}(\chi(t))|\). This is described in section \ref{Sec4}. In subsection \ref{SubsecModel}, we propose an algorithm for deriving the function \(g\) and parameters \((\mu_j)_{j=0}^m\) such that \(\sum_{i=0}^{n}\frac{d_i}{d_1} r_{1}\) appropriately estimates the envelope of \(\chi(t)\). We apply our proposed modelling procedure on the pitch \(C\sharp 4\) played by piano and violin. This is to illustrate the differences in qualitative changes (timbral behavior) of a musical note in different musical instruments. This is reflected in different bifurcation scenarios. Our proposed approach can be similarly applied on other musical notes and musical instruments. Figure \ref{Csharp4Plot} shows graphical representation of  note $C\sharp 4$ performed by piano and violin. These simulated timbral notes experience different sets of bifurcation types.

\subsection{Dynamics and temporal envelop }\label{SubsecModel}

Temporal envelopes are generally split into several time-segment types; namely, delay, attack, hold, decay, sustain and release. We determine these time segments by introducing the {\it border points} from the time interval. The time interval between any two consecutive border points is characterised by one of these six time-segment types. In order for sufficiently accurate simulation of temporal envelops, we refine these time-segments by further introducing {\it breaking points} within these time-segments.
In other words, each time-segment interval is divided into several subintervals by the {\it breaking points} to increase the accuracy of the estimation; see Algorithm \ref{Algorthm} and  Figure \ref{D4-Csharp4Plot}, where the yellow bullets show border points while the black bullets denote the breaking points.
Let
\ba\label{gForm}
&g(\x, \y):=a (\sum_{i=0}^{n}\frac{d_i}{d_1})^2\,({x_1}^2+{y_1}^2)+b(\sum_{i=0}^{n}\frac{d_i}{d_1})^4{({x_1}^2+{y_1}^2)}^2, &
\ea in equation \eqref{SimDyn}. Hence, the reduced amplitude equation on the manifold \(\mathcal{M}_{{\pmb{c}}}\) on subinterval \(t_j\leq t\leq t_{j+1}\) is given by
\begin{eqnarray}\label{eq4.5}
&\frac{\rm d}{{\rm d}t}\mathbf{r}=\frac{\alpha\,r_1}{c_1}{\pmb{c}}\left(\mu_j+a \left(\sum_{i=0}^{n}\frac{d_i}{d_1}\right)^2\,{r_1}^2+b\,\left(\sum_{i=0}^{n}\frac{d_i}{d_1}\right)^4{r_1}^4\right).&
\end{eqnarray}
Therefore, we introduce \(\rho:=\sum_{i=0}^{n}\frac{d_i}{d_1}r_1(t, \mu)\) and simulate the envelope of \(\chi(t)\) on each subinterval by tuning the parameter \(\mu_j\) via the following scalar equation
\begin{eqnarray}\label{rhoDot}
&\frac{\mathrm{d} \rho}{\mathrm{d}t}=\alpha \rho(\mu_j +a{\rho}^2+b {\rho}^4).&
\end{eqnarray}
Here, the coefficients \(a\) and \(\alpha\) are merely two constants for all time-segment intervals and subintervals while \(\mu_j\) stands for the bifurcation controller (tuning) parameter for the \(j\)-th subinterval. As the constant \(\alpha\) increases, the slope of the solution trajectory increases and vice versa. This gives rise to an effective controller design via the proportional correspondence between \(\alpha\) and the slope of the temporal envelop. A sign change in \(\alpha\) results in a bifurcation of the stability types, that is, all stabilities of invariant tori and the origin change. The bifurcation controller parameter \(\mu_j\) plays a central role for the appearance/disappearance of invariant hypertori and sizes of their corresponding radiuses. The convergence speed of trajectories converging to or diverging from an invariant hypertorus can be directly controlled by \(|\mu_j|\) and \(|\alpha|\); see Figure \ref{BifDiag}.
Further, 
\begin{enumerate}
\item\label{1} (Continuous property of the solution trajectory) The initial value for the differential system \eqref{rhoDot} on each time subinterval must be taken as the same as the end value of solution trajectory corresponding with the previous subinterval.
\item The parameter \(\mu_j\) is tuned to match the slope of the right side tangent line (right derivative) on \(r_1(t, \mu)\) at the border and breaking points with those of the temporal envelop of the simulating note.
\item The subintervals are refined by additional breaking points to satisfy the condition \eqref{con2}.
\end{enumerate}
Our proposed required properties for a proper simulation using Eulerian differential systems \eqref{SimDyn}-\eqref{gForm} in the time-segments associated with the temporal envelope are listed  as follows:
\begin{enumerate}[label=(\alph*)]
\item\label{I_Attack} For the case of {\it attack}, the differential system must have an unstable equilibrium at the origin while the initial values of the solution
starts close to the origin and are non-zero. It should be small enough to be less than the hearing normalised velocity threshold. The amplitude growth speed of the solution trajectory can be controlled by \(\mu_j\). The differential system may have an asymptotically stable flow-invariant \(n\)-torus, where the attack trajectory approaches.
\item \label{I_delay}It suffices for the {\it delay time-subinterval case} that the origin would be an asymptotically stable equilibrium while the nonzero initial values is small enough so that the simulated sound would be translated into an inaudible sound.
\item \label{I_decayRelease} Simulations of time-subintervals for the {\it decay } and {\it release} either require an asymptotically stable \(n\)-torus with initial
values starting outside of the hypertorus or an asymptotically stable origin attracting the trajectories.
\item\label{I_sustain}
We consider the cases when we have at least a two-consecutive steady-state sustain segment; for example see the sustain segment in Figure \ref{Csharp4V}. These cases require two flow-invariant hypertori, where one of them is stable and the other is unstable. The unstable torus must be designed so that it is close to the initial values. Then, the trajectory diverges from the unstable hypertorus and converges to the stable hypertorus as the time runs in the sustain time-segment. When the envelope consists of only one steady-state sustain segment or hold segment, it suffices to take the initial values near or on an asymptotically stable hypertorus.
\end{enumerate}

\section{An algorithm for implementation} \label{Algorthm}

Our objective here is to introduce an algorithm for an accurate estimation of the temporal envelop by computing the breaking points, the constants \(a,b,\) \(\alpha\) and tuning parameters \(\mu_j.\)
\begin{enumerate}[label=\textnormal{(\roman*)}]
\item(Numerical computation of the upper envelope) \label{I_Coeff_1} First we get an approximation of upper envelope for each portion of the sound signal by the
command $\gamma$ = \texttt{envelope(x, np,'peak')} in {\sc Matlab}. This command in {\sc Matlab} uses the spline interpolation over local maxima, where they are  separated by at least $\texttt{np}$ samples. The curve $\gamma(t)$ is called by the upper envelope.
\item\label{I_coeffA}(Computing the constants \(a, b \) and \(\alpha\)) We consider two cases. The first case refers to when the \textit{sustain} segment of the upper envelope have at least a two-consecutive steady-state sustain segment. Otherwise, we consider \(\alpha:=1, a:=-1\) and \(b:=0.\) For the first case, let \(a:=1\) (\(a:=-1\)) for the increasing (the decreasing) sustain level. Also assume that the sustain segment be the ${j_s}$-th time-segment of the envelope. Then, for \(t=t_{{j_s}}\) and \(t=t_{{j_s}+1},\) we take \(\frac{\textrm{d}}{\textrm{d} t}\rho(t, \mu_{j_s})\approx 0.\) Therefore,
\begin{eqnarray}\label{calValue b}
   \mu_{j_s}+a(\gamma(t_{{j_s}+1})-\epsilon)^2+b(\gamma(t_{{j_s}+1})-\epsilon)^4=0,\quad \mu_{j_s}+a{\gamma(t_{{j_s}})}^2+b{\gamma(t_{{j_s}})}^4=0.
\end{eqnarray} This gives rise to the required values for \(b\) and \(\mu_{j_s}.\) The parameter \(\alpha\) is derived so that the slopes of the curves \(\gamma\) and \(\rho\) would be the same at the middle of the curve \(\gamma\) over the interval \((t_{j_s}, t_{{j_s}+1})\). Yet, a \(\epsilon\) (\(0<\epsilon<0.05\)) is subtracted from \(\gamma(t_{{j_s}})\) to locate the associated root \(\gamma(t_{{j_s}})\) slightly lower than \(\rho(t_{j_s}^-, \mu_{j_s}).\) Hence, \(\alpha\) is calculated from equation \eqref{rhoDot} and
\begin{eqnarray}\label{calValue alpha}
&\frac{\rm d}{{\rm d} t}\rho(\gamma^{-1}((\gamma(t_{{j_s}})-\epsilon+\gamma(t_{{j_s}+1}))/2), \mu_{j_s})= \dot{\gamma}(\gamma^{-1}((\gamma(t_{{j_s}})-\epsilon+\gamma(t_{{j_s}+1}))/2)).&
\end{eqnarray} The constants \(a\) and \(\alpha\) are fixed for the entire coloring of a given note while the static controller parameter \(\mu_j\) is varied as \(j\) changes. Thus, the constants \(a\) and \(\alpha\) are known for \(j\neq j_s\).
\item\label{I_rho0} For delay interval \((t_0, t_1),\) take \(\mu_0\leq-|\frac{1}{4b}|\) and \(\rho(0, \mu_0)\leq\min\{0.01, |2b|^{-\frac{1}{2}}\}.\) For other
segments, we follow items \ref{I_Coeff_2} and \ref{I_Coeff_3}, see below.
\item \label{I_Coeff_2}
When there is no breaking points on the interval \([t_{j-1}, t_j]\), take the initial values \(\rho(t_{j}, \mu_j)=\rho(t_{j}^-, \mu_{j-1})\) and \(\dot{\rho}(t_{j}^+, \mu_j)=\dot{\gamma}(t_{j}^+).\) Here, \(\dot{\gamma}(t_{j}^\pm)\) and \(\gamma(t_{j}^\pm)\) denote the classical right (left)-hand derivative and limits of \(\gamma\) from right (left) at \(t=t_{j}.\) Therefore,
\begin{eqnarray}\label{calValue Mu}
  &\dot{\rho}(t_{j}^+, \mu_j)=\alpha \rho(t_{j}^-, \mu_{j-1})(\mu_j+a{\rho(t_{j}^-, \mu_{j-1})}^2+b{\rho(t_{j}^-, \mu_{j-1})}^4)=\dot{\gamma}(t_{j}^+).&
\end{eqnarray} From equation \eqref{calValue Mu}, the constant parameter \(\mu_j\) is calculated.
  \item(The breaking point computations) \label{I_Coeff_3} As long as the error between \(\rho\) and \(\gamma(t)\) is larger than our specified
threshold, we add the mid-points of the intervals into the breaking points. More precisely, the extra breaking points belong to the interval \([t_j, t_{j+1}]\) and are denoted by \(t_j^l\) where
\begin{eqnarray*}
t_j^l:=\gamma^{-1}((\gamma(t^{l-1}_j)+\gamma(t_{j+1}))/2^m),
\end{eqnarray*} and \(m\in \mathbb{Z}^{\geq 0}\) is the minimum number that satisfies
\begin{eqnarray}\label{con3}
|\rho-\gamma(t)|< 0.05 \quad \text{ for all }\quad  t_{j}^{l-1}\leq t\leq t_{j+1}^{l}.
\end{eqnarray} When there is no breaking point within an interval \([t_j, t_{j+1}]\) or \([t_j, t_{j}^1],\) the calculated value for \(\mu\) is simply denoted by \(\mu_j.\) However, the allotted static constant \(\mu\) corresponding with the time interval \([t^l_j, t^{l+1}_j]\) or \([t^l_j, t_{j+1}]\) is denoted by \(\mu^l_j.\) The parameter \(\mu^l_j\) is similarly calculated via equation \eqref{calValue Mu} by evaluating it at \(t=t_j^l\). A smooth assumption for the upper envelop signal \(\gamma(t)\) guarantees that a finite set of breaking points is sufficient for a {\sc Matlab} program to estimate \(\gamma(t)\) with \(\rho\) with an error less than the specified threshold \eqref{con3}. We test the condition \eqref{con3}, using the following {\sc Matlab} command
\begin{eqnarray}\label{Eq4}
&\underset{t_j^l\leq t\leq t_j^{l+1}}{\max}\left\vert\rho(t, \mu_j^l)-\gamma(t)\right\vert=\texttt{max(abs(($\rho$(Fs$\cdot$($t_1$:$t_2$))-$\gamma$(Fs$\cdot$($t_j^l$:$t_j^{l+1}$)))))}. &
\end{eqnarray} When the right-hand value of \eqref{Eq4} is less than \(0.05,\) the condition \eqref{con3} is satisfied.

\end{enumerate}
\subsection{Analysis of the algorithm}

Consider equation \eqref{rhoDot} and the case \ref{I_sustain} where sustain segment is increasing and consists of two successive steady-states.
Hence, \(a=1\) and equation \eqref{calValue b} is solved for the parameter \(\mu\) and constant \(b.\) Then, equilibria are   \(\smash{\big(\frac{1\pm\sqrt{1-4b\mu}}{-2b}\big)\!^{\frac{1}{2}}}\)  and \(-\smash{\big(\frac{1\pm\sqrt{1-4b\mu}}{-2b}\big)\!^{\frac{1}{2}}}.\) This equation has two positive roots when  \(b<0\) and \(0<b\mu<\frac{1}{4},\) \ie \(\mu<0.\) Since \(\mu\) is negative, the smaller equilibrium is unstable and the larger one is asymptotically stable. Therefore, a trajectory \(\rho(t, \mu)\) approaches to the larger root when it starts from somewhere between the two roots. When the sustain segment involves two decreasing consecutive steady-states, \(a=-1.\) Then, equation \eqref{calValue b} results in positive values for \(\mu\) and \(b.\) Hence, the smaller root is asymptotically stable and the larger one is unstable. The smaller root is chosen so that initial value is placed between the two roots (but near to the larger root). Thereby, the conditions of item \ref{I_sustain} is satisfied. The parameter \(\epsilon\) is accommodated in equation \eqref{calValue b} for this purpose.

When \(b<0,\) \(a=1\), and \(\mu_0<\mu^{\star}=-|\frac{1}{4b}|\), the origin is asymptotically stable and its basin of attraction is the whole state space; see Figure \ref{SubP2SN} and Theorem \ref{lem3}. These correspond with the delay time segment.
For the case \(b>0,\) \(a=-1\), and \(\mu_0<0,\) the scalar  differential equation \eqref{rhoDot} has one positive equilibrium \(\smash{\big(\frac{1+\sqrt{1-4b\mu_0}}{2b}\big)\!^{\frac{1}{2}}}\); see Figure \ref{SupP2SN}. This is larger than the initial value \(\rho(0, \mu_0)\leq\min\{0.01, |\frac{1}{2b}|^{\smash{\frac{1}{2}}}\}.\) Hence, our expectation in item \ref{I_delay} holds.

In the attack interval, we must have \(\dot{\rho}(t, \mu_1)>0.\) For \(a=1\) and \(b<0,\) the condition \(\dot{\rho}(t, \mu_1)>0\) holds when (1) \(\mu_1<0\) and \(\rho(t_1^+, \mu_1)\) is between the two equilibria or (2) \(\mu_1>0\) and \(\rho(t_1^+, \mu_1)\) is between the origin and a positive equilibrium; see Figure \ref{SubP2SN}. Since \(\rho(t_1^+, \mu_1)<\rho(0, \mu_0)\leq 0.01,\) \(\rho(t_1^+, \mu_1)\) is less than the threshold and sufficiently close to the origin. For \(a=-1\) and \(b>0,\) the condition \(\dot{\rho}(t, \mu_1)>0\) is valid when (1) \(\mu_1>\frac{1}{4b}\) or (2) \(0<\mu_1<\frac{1}{4b}\) and  \(\rho(t_1^+, \mu_1)\) is between the origin and the stable equilibrium; see Figure \ref{SupP2SN}. Since \(\rho(t_1^+, \mu_1)< |\frac{1}{2b}|^{\smash{\frac{1}{2}}}\), we have \(\mu_1>0.\) The condition in item \ref{I_Attack} is fulfilled because \(\rho(t_1^+, \mu_1)\) is less than the threshold.

For the decay and release segments, we must have \(\dot{\rho}(t, \mu)<0.\) This condition for \(a=1\) and \(b<0\) holds if
(1) \(\mu<\frac{-1}{4b},\) (2) \(\frac{-1}{4b}<\mu\) and the initial value is more than the largest (stable) equilibrium, or (3) \(\frac{-1}{4b}<\mu\) and the initial value is between the origin and the unstable equilibrium. For \(a=-1\) and \(b>0,\) the condition \(\dot{\rho}(t)<0\) is satisfied when \(\frac{1}{4b}>\mu\) and the initial value either is between the two equilibria or is between the origin and unstable equilibrium. Hence, the conditions of \ref{I_decayRelease} are satisfied.
For simulation of sustain and hold segments, parameter \(\mu\) is such that sustain or hold-levels are equilibria of \eqref{rhoDot}.
By similar reasonings, conditions \ref{I_Attack}, \ref{I_delay}, \ref{I_decayRelease},  and \ref{I_sustain} are proved to hold for the case \(a:=-1\) and \(b:=0.\)



\section{Differential systems for pitch $C\sharp 4$}\label{SecPianoCons}

We consider an audio file of the pitch $C\sharp 4$ played with piano. The waveform of this audio signal is plotted in Figure \ref{Csharp4P} using {\sc Matlab} command ``\texttt{audioread}''; see Remark \ref{Rem4.12}. Here, the piano key $C\sharp 4$ is pressed at $0.083 \sec$ and is released at $0.293 \sec.$ Temporal envelope of this musical tone consists of delay time, attack time, decay time, sustain level and release time. These correspond with the intervals \([0, 0.084]\), \([0.084, 0.112]\),\( [0.112, 0.160],\) \([0.160, 0.293]\) and \([0.293, 3.5],\) respectively. Thus, the time interval is divided by the border points \(t_0=0, t_1=0.084, t_2=0.112, t_3=0.16, t_4=0.293, t_5=3.5.\) Next, we apply Algorithm \ref{Algorthm} on these time intervals. This process adds some extra breaking points as described by item \ref{I_Coeff_3}. Thereby,  the time interval \([t_0, t_5]\) is divided by
\be\label{ti1} t_0, t_1, t_1^1:=0.066, t_1^2:=0.089, t_1^3:=0.0991, t_2, t_3, t_4, t_4^1:=0.467, t_4^2:=0.6785, t_5.\ee
Now we explain the details. Since the amplitude envelope function \(\gamma(t)\) only contains one steady-state segment during the sustain level, we take \(a:=0\) in equation \eqref{rhoDot}; see item \ref{I_coeffA} of the algorithm. In the delay (silence) interval for all \(t\in (t_0, t_1)\), \(\gamma(t)=0.\) For this case, we take \(\rho(0, \mu_0):= 0.001<0.01,\) \(\alpha_0:= -\mu_0:=1>0\) so that the condition \eqref{con3} is satisfied.
By solving equation \eqref{rhoDot}, we have \(\rho(t_1, \mu_1)=8.75\times 10^{-4}.\) In the case of attack interval,
the time interval corresponding with \(\gamma^{-1}((\gamma(t_1)+\gamma(0.0991))/2^i\) for \(i=0,1,2\) are not sufficient to satisfy the condition \eqref{con3}. Thus, we accommodate the first extra breaking point as \(t_1^1=\gamma^{-1}((\gamma(t_1)+\gamma(0.0991))/8)=\gamma^{-1}(0.066)=0.0865\) where \(\gamma(t_1)=0\) and \(\gamma(0.0991)=0.528.\) Therefore, (by item \ref{I_Coeff_2})
\begin{eqnarray*}
&\scalebox{0.97}{$\frac{{\rm d} \rho}{{\rm d}t}({t_1}^+, \mu_1)= \rho(t_{1}, \mu_1)(\mu_1-{\rho(t_{1}, \mu_1)}^2)=\dot{\gamma}({t_{1}}^+)=1.66$} \quad \hbox{ and } \quad \mu_1:=1902.&
\end{eqnarray*} Then, the condition \eqref{con3}  for \(t_1\leq t\leq t_1^1\) is satisfied due to the inequality
\begin{eqnarray*}
&{\max}_{t_{1}< t \leq t_1^1}\big|\rho(t, \mu_1)-\gamma(t)\big|=0.03<0.05.&
\end{eqnarray*}
Now we include a breaking point at \(t_1^3:=0.0991\) while the second extra breaking point is given by \(t_1^2:=\gamma^{-1}((\gamma(t_1^1)+\gamma(t_1^3))/2)=\gamma^{-1}(0.297)=0.089.\) There is no breaking point within \((t_1^2, t_1^3)\) and \((t_1^3, t_2)\). Similarly, the remaining attack segment parameters are derived as \(\mu_1^1:=498,\) \(\mu_1^2:=220,\) and \(\mu_1^3:=14.\) In the case of decay time, we have \(\rho(t_2, \mu_2)=0.720,\) \(\dot{\gamma}({t_{2}}^+)=-7.2,\) and $\frac{{\rm d} \rho}{{\rm d}t}({t_2}^+, \mu_2)= \rho(t_{2}, \mu_2)(\mu_2+{\rho(t_{2}, \mu_2)}^2)=-7.2.$ Hence, \(\mu_2=-9.5.\) For the sustain level, \ie \(0.16\leq t\leq 0.301,\) \(\mu_3:=0.28.\) This is due to \(\rho(t_3, \mu_3)=0.519,\) \(\dot{\gamma}({t_{3}})=-0.005,\) and \(\frac{{\rm d} \rho}{{\rm d}t}({t_3}^+, \mu_3)= \rho(t_{3}, \mu_3)(\mu_3-{\rho(t_{3}, \mu_3)}^2)=-0.005\).
For the release interval, we have the breaking point \(t_4^1:=\gamma^{-1}((\gamma(t_4)+\gamma(t_5))/2) =\gamma^{-1}(0.2692)=0.467.\) Now from item  \ref{I_Coeff_2},
\(\rho(t_4, \mu_4)=0.5166 \) and
\begin{eqnarray*}
&\frac{{\rm d} \rho}{{\rm d}t}({t_4}^+, \mu_4)= \rho(t_{4}, \mu_4)(\mu_4-{\rho(t_{4}, \mu_4)}^2)=\dot{\gamma}({t_4}^+)=-2.4.&
\end{eqnarray*}
 Thereby, \(\mu_4:=-4.5.\) Let \(t_4^2:=\gamma^{-1}((\gamma(t_4^1)+\gamma(t_5))/2) =\gamma^{-1}(0.14)=0.6785.\) From
\begin{eqnarray*}
&\frac{{\rm d} \rho({t_4^1}^+, \mu_4^1)}{{\rm d}t}= \rho(t_{4}^1, \mu_4^1)(\mu_4^1-{\rho(t_{4}^1, \mu^1_4)}^2)=\frac{\textrm{d} \gamma(t_4^1)}{\textrm{d} t}=-0.78 \quad \hbox{ and } \quad \rho(t_4^1, \mu_4^1)=0.2418,&
\end{eqnarray*} we have \(\mu_4^1:=-3.2.\) Due to \(\rho({t_4^2}, \mu_4^2)=0.2418\) and \(\frac{{\rm d} \gamma(t_4^2)}{{\rm d}t}=-0.09,\) we have \(\mu_4^2=-0.8\) for the subinterval \((t_4^2, t_5).\) This completes the determination of all parameters and constants of the Eulerian differential system \eqref{SimDyn} (\(g\) is given by equation \eqref{gForm}) associated with the modeling of an audio signal $C\sharp4$ played by piano. By the obtained values in subsection \ref{SecSpectrm}, we have \(\sum_{i=0}^{6}d_i/ d_1:= 2.364,\)
\begin{eqnarray*}
&n=6,\quad\pmb{\omega}=\pmb{\omega}_{\rm piano}^{C\sharp4},\quad {\pmb{c}}={\pmb{c}}_{\rm piano}^{C\sharp4},\quad a=0, \quad \alpha=1, \quad g(\x, \mathbf{y})= -5.588({x_{1}}^2+{y_{1}}^2).&
\end{eqnarray*} The time segments and breaking points are given in the list \eqref{ti1}. The initial values are given by \(\x(0, \mu_0)=10^{-4}(-5.74, 4.23, 3.65, 2.39, 1.62, 2.21, 1.63)\) and \(\mathbf{y}(0, \mu_0)=\0\) while
\begin{eqnarray*}
  \mu(t)\!&\!:=\!&\!-H(t)+1903H(t-t_1)-1404H(t-t_1^1)-278H(t-t_1^2)-206H(t-t_1^3)\\
  &&-23.5H(t-t_2)+9.78H(t-t_3)-4.78H(t-t_4)+1.3H(t-t_4^1)-2.4H(t-t_4^2).
\end{eqnarray*} For the simulated signal \(\sum_{i=0}^{6} x_i\) of the piano waveform, the first six ratios of the peaks of amplitude partials to the fundamental frequency amplitude is \((1, 0.8619,0.566, 0.376, 0.514, 0.375).\) These are very close to \((1 , 0.8634 , 0.565, 0.3845, 0.522, 0.38)\) associated with the actual piano audio signal.

\begin{figure}[t]
\centering
\subfloat[Pitch $C\sharp4$ waveform played by  piano \label{CsharpP_Sim}]{\includegraphics[width=3.6in]{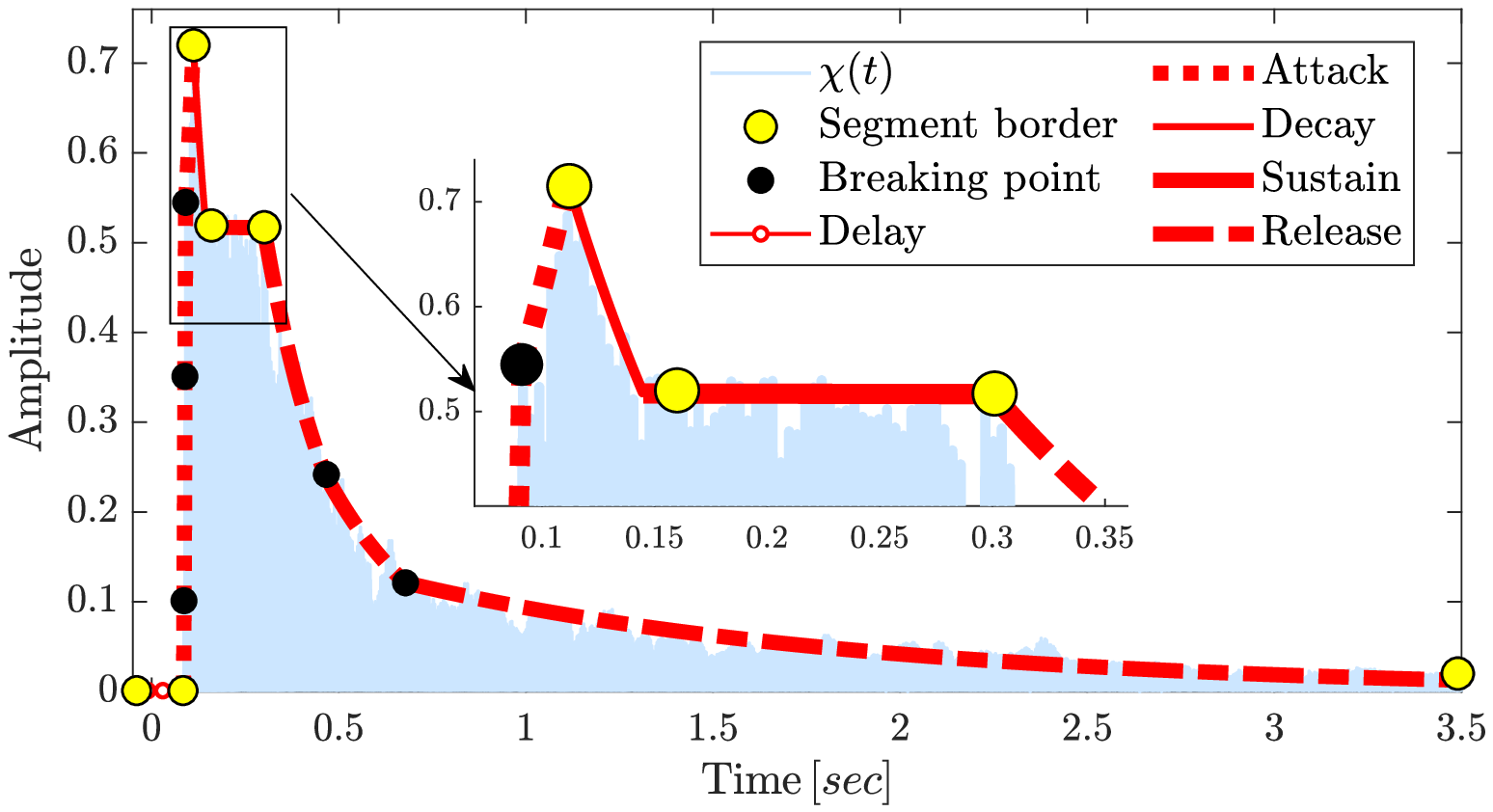}}
\subfloat[Pitch $C\sharp4$ waveform played by violin \label{CsharpV_Sim}]{\includegraphics[width=3.6in]{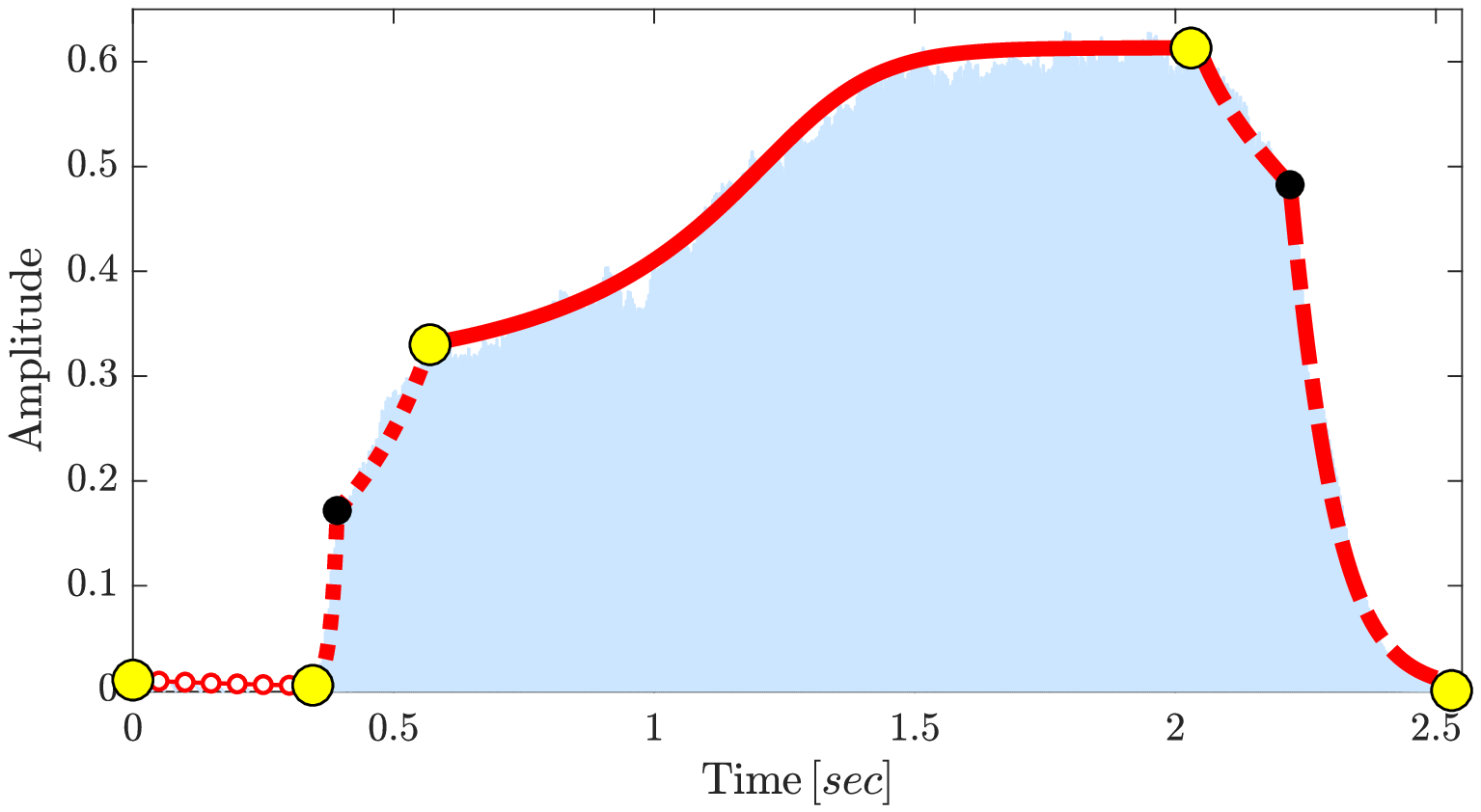}}
\caption{Envelope estimation for the audio file obtained from playing $C\sharp4$ by a piano and a violin.}\label{D4-Csharp4Plot}
\end{figure}

Now consider a $C\sharp 4$ waveform in Figure \ref{Csharp4V}. This is obtained from an audio signal file of a $C\sharp 4$ played with violin. The violin is bowed at \(0.35\sec\) and the bowing is stopped at \(2.107\sec\). The temporal envelope in this figure is derived by {\sc Matlab}. This consists of a delay time, attack time, sustain level and release time. The segment borders and two extra breaking points make a partition for the time interval as
\be\label{ti2}
t_0=0, t_1:=0.345, t_1^1:=0.392, t_2:=0.5717, t_3:= 2.107, t_3^1:= 2.22, t_4:=2.5.
\ee Following Algorithm \ref{Algorthm}, we first explain the sustain level interval. The sustain level includes two steady-states at the beginning and at the end of the sustain segment. Thus,
\begin{eqnarray*}
\mu_2+(\gamma(t_{2})-0.03)^2+b(\gamma(t_{2})-0.03)^4=0,\quad \mu_2+{\gamma(t_{3})}^2+b{\gamma(t_{3})}^4=0,
\end{eqnarray*} where \(\gamma(t_1)=0.331, \gamma(t_2)=0.622.\) Hence, \(b:=-2.15\) and \(\mu_2:= -0.072.\)  By item \ref{I_coeffA},
\begin{eqnarray*}
&\frac{{\rm d} }{{\rm d}t} \rho(\gamma^{-1}(\gamma(t_{2})-0.03+\gamma(t_{3}))/2, \mu_2)= \dot{\gamma}(\gamma^{-1}(\gamma(t_{2})-0.03+\gamma(t_{3}))/2) \quad\;  \hbox{ and } \; \alpha:=23.&
\end{eqnarray*} In the delay segment, we let \(\rho(t_0, \mu_0)=0.01\) and \(\mu_0:=-0.1\) while the attack duration takes an extra breaking point given by \(t_1^1=\gamma^{-1}((\gamma(t_1)+\gamma(t_2))/2) =0.392.\) From items \ref{I_Coeff_2}, \ref{I_Coeff_3}, \(\rho(t_1, \mu_1)=0.0045,\) and \(\frac{{\rm d} \rho(t_1^+, \mu_1)}{{\rm d}t}=\frac{{\rm d} \gamma(t_1^+)}{{\rm d}t}=0.35,\) we have \(\mu_1:=3.36\) and the condition \eqref{con3} is satisfied. For the interval \((t_1^1, t_2),\) \(\mu_1^1:=0.11\) by a similar argument. In the sustain stage, we have already obtained the value \(\mu_2:= -0.072\). In the release time, we have the breaking point \(t_3^1=\gamma^{-1}((\gamma(t_3)+\gamma(t_4))/4)=2.22\), \(\frac{{\rm d} \rho}{{\rm d}t}(t_3^+, \mu_3)=\dot{\gamma}({t_3}^+)=-1.31\) and \(\rho(t_3, \mu_3)=0.6132\). Hence, \(\mu_3:=-0.165.\) Finally, the remaining parameter \(\mu_3^1:=-0.6\) and is due to the time interval \((t_3^1, t_4)\). The dynamical models of the pitch \(C\sharp4\) played by violin is expressed by the initial value problem \eqref{SimDyn} and \(g\) given by equation \eqref{gForm}. Now we have \(\sum_{i=0}^{6}d_i/ d_1=1.1246,\) equations \eqref{ti2} hold, and
\begin{eqnarray*}
&n=6,\, \pmb{\omega}=\pmb{\omega}_{\rm violin}^{C\sharp4},\, {\pmb{c}}={\pmb{c}}_{\rm violin}^{C\sharp4},\, a=-2.15, \, \alpha=23, \, g(\x, \mathbf{y})= 1.578 ({x_{1}}^2+{y_{1}}^2)-5.35({x_{1}}^2+{y_{1}}^2)^2.&
\end{eqnarray*}  The initial values are given by \(\x(0, \mu_0)=10^{-4}(-23, 61, 22, 7, 3, 2, 5)\) and \(\mathbf{y}(0, \mu_0)=\0\). The piecewise constant bifurcation parameter \(\mu\) (\ie quasi-statically varies) follows
\begin{eqnarray*}
&\mu(t):=-0.1H(t)\!+\!3.46H(t-t_1)\!-\!3.25H(t-t_1^1)\!-\!0.182H(t-t_2)\!+\!0.237H(t-t_3)\!-\!0.765H(t-t_3^1).&
\end{eqnarray*}
\begin{figure}[t]
\centering
\includegraphics[width=7.2in]{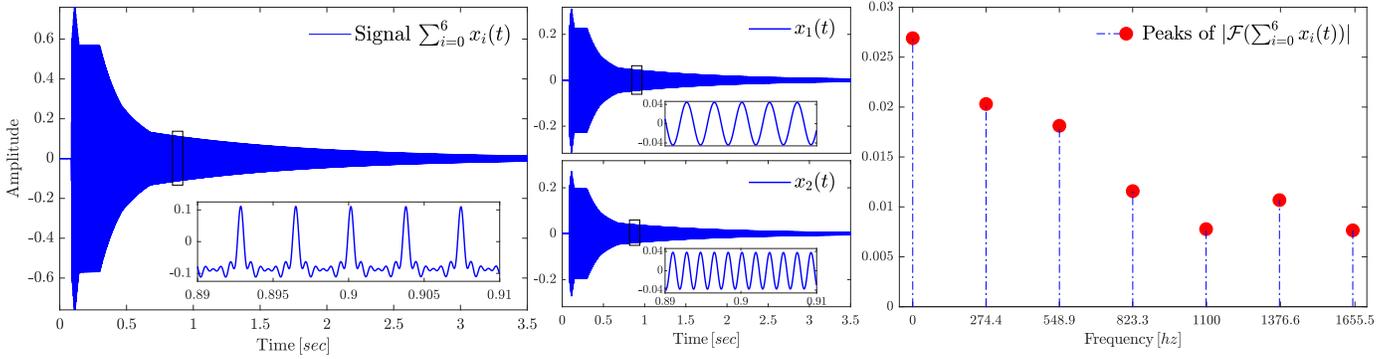} 
\caption{Simulated additive synthesis \(\sum_{i=0}^{6} x_i,\) \(x_1, x_2\) and peaks for $C\sharp 4$ played by piano }\label{Sim_PianoCsharp4}
\end{figure}
The first six ratios of the leading partial amplitudes to the fundamental amplitude for the audio violin signal are
\((1, 0.362, 0.1124,  0.0513,  0.0318,  0.0825)\) while \((1,  0.362, 0.1121, 0.0513, 0.0314, 0.0813)\) are those for the simulated violin data corresponding with \(\sum_{i=0}^{6} x_i(t)\).


\section{Leaf bifurcations of \(C\sharp 4\) played by piano and violin}\label{SecleafBifurcations}

In this section we explain how a consecutive ordered set of four different type of pitchfork bifurcations describe the audio $C\sharp 4$ file generated by piano. The description for the ordered set of bifurcations associated with the audio $C\sharp 4$ played by violin include only two consecutive pitchfork bifurcations followed by a double saddle-node bifurcation. The latter constitutes a complete hysteresis cycle for the scalar differential equation. This illustrates an audio sample of hysteresis type phenomena involving a complete hysteresis cycle and hypertori. Recall that the pitchfork and saddle nodes of bifurcations are merely associated with the reduced scalar equation. However, they infer about the existence and disappearance of the flow-invariant Clifford hypertori for the Eulerian differential system over the musical note's duration. 

The bifurcation point for the $C\sharp 4$ waveform generated by piano is \((\mu,r_1 )=(0, 0).\) Bifurcations occur at the border points \(t_1, t_2, t_3\) and \(t_4\) while there is no any type of bifurcation at the breaking points. By Theorem \ref{Bifurcations}, each of these correspond with a scalar supercritical {\it pitchfork} type bifurcation due to the change in the sign of the parameter \(\mu\). In the delay interval \((t_0, t_1),\) the sign of \(\mu_0\) is negative. Therefore, there is no flow-invariant hypertorus for the differential system, the origin is stable and the trajectories remain close to the origin. On the attack interval \((t_1, t_2),\) there is one stable flow-invariant hypertorus \(\mathbb{T}_6.\) This is because the parameters \(\mu_1, \mu_1^1,\) and \(\mu_1^2\) are all positive. In the decay interval \((t_2, t_3),\) \(\mu_2<0\) and thus, the hypertorus disappears and the origin becomes asymptotically stable. For the sustain level, \(\mu_3\) is close to \(13.84\,{r_1(t_3^-)}^2\). Thus, there is an asymptotically stable hypertorus and the trajectories remain close to the hypertorus over the interval \((t_3, t_4).\) Since \(\mu_4<0\) and \(\mu_4^1<0\), trajectories tend to the origin over the release time \((t_4, t_5)\).

\begin{figure}[t]
\centering
\includegraphics[width=7.2in]{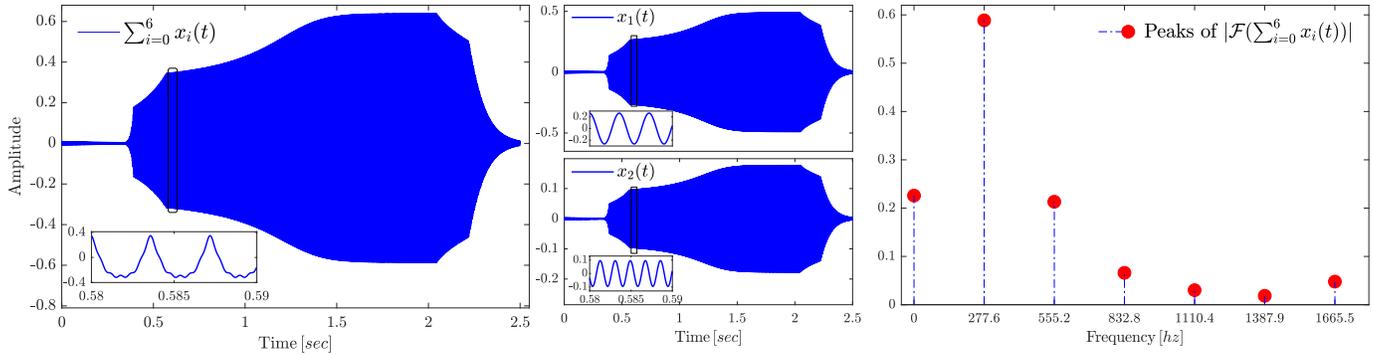} 
\caption{Additive synthesis and simulations for the audio pitch $C\sharp 4$ played by violin }\label{Sim_PianoCsharp4}
\end{figure}

Now consider the waveform obtained from playing \(C\sharp 4\) by violin. For the delay duration \((t_0, t_1)\), \(\mu_0<0\) while \(\mu_1, \mu_1^1\) are both positive over attack time interval \((t_1, t_2)\). This change of sign for \(\mu\) at time \(t_1 \) corresponds with a subcritical {\it pitchfork} type bifurcation for the scalar equation \eqref{rhoDot}. By Lemma \ref{lem3}, this means that one flow-invariant hypertorus \(\mathbb{T}_6\) disappears when both positive and negative roots of the scalar equation collide at zero when \(t=t_1\). Therefore, there is an unstable flow-invariant hypertorus \(\mathbb{T}_6\) over attack time interval \((t_1, t_2)\) while there are two flow-invariant hypertori \(\mathbb{T}_6\) over the delay interval \((t_0, t_1)\). In the delay case, the internal flow-invariant hypertorus is unstable while the origin and the external hypertorus are asymptotically stable.

In the time interval corresponding with the sustain level \((t_2, t_3),\) \(-0.116<\mu_1<0\) and one local hypertorus come again to existence. Recall that we had an asymptotically stable 6-torus in the attack time and a new unstable 6-torus is born from the origin when time passes through \(t_2.\) This is associated with a subcritical {\it pitchfork} type of bifurcation for the scalar equation \eqref{rhoDot}. For the release interval \((t_3, t_4),\) \(\mu_1<-0.116\) and both of the hypertori disappear. In terms of Lemma \ref{lem3}, \(T_{2SN}=\{\mu \,|\, \mu= -0.116\}\) is a double saddle-node bifurcation transition variety for the scalar equation \eqref{rhoDot}. These saddle-node bifurcations occur at \((\mu, r_1)=(-0.116, \pm 0.5)\). Since only two of the roots are positive, there are only two hypertori over \((t_2, t_3)\) and they disappear when they collide with each other at \(t_3.\) Hence, there is no hypertorus over the time interval \((t_3, t_4).\)


\end{document}